\documentclass[12pt,a4paper]{amsart}
\usepackage{amssymb}
\usepackage{amsfonts}
\usepackage{amsthm}
\usepackage{amsmath}
\usepackage{amscd}
\usepackage[latin2]{inputenc}
\usepackage{t1enc}
\usepackage[mathscr]{eucal}
\usepackage{indentfirst}
\usepackage{graphicx}
\usepackage{graphics}
\numberwithin{equation}{section}
\usepackage[margin=2.9cm]{geometry}
\usepackage{epstopdf} 
\usepackage{xcolor}
 
\usepackage{verbatim}
 \usepackage{enumitem}
\usepackage{tikz}
\usepackage[percent]{overpic}

\theoremstyle{plain}
\newtheorem{theorem}{Theorem}[section]
\newtheorem{lemma}[theorem]{Lemma}
\newtheorem{corollary}[theorem]{Corollary}
\newtheorem{proposition}[theorem]{Proposition}

 \theoremstyle{definition}
 \newtheorem{condition}[theorem]{Condition}

\newtheorem{remark}[theorem]{Remark}
\newtheorem{?}[theorem]{Problem}

\title[Inverse Problems for Single Photons]{Inverse Scattering for Single Photons in Quantum Optics}

\author{}
\address{}
\email{}

\date{}



\author[Covi]{Giovanni Covi}
\address{Department of Mathematics and Statistics, University of Jyv\"askyl\"a, Jyv\"askyl\"a, Finland}
\email{giovanni.g.covi@jyu.fi}

\author[Lassas]{Matti Lassas}
\address{Department of Mathematics and Statistics, University of Helsinki, Helsinki, Finland}
\email{matti.lassas@helsinki.fi}

\author[Nursultanov]{Medet Nursultanov}
\address{Qazaq AI Research University, Astana, Kazakhstan}
\email{medet.nursultanov@gmail.com}

\author[Oksanen]{Lauri Oksanen}
\address{Department of Mathematics and Statistics, University of Helsinki, Helsinki, Finland}
\email{lauri.oksanen@helsinki.fi}

\author[Salo]{Mikko Salo}
\address{Department of Mathematics and Statistics, University of Jyv\"askyl\"a, Jyv\"askyl\"a, Finland}
\email{mikko.j.salo@jyu.fi}

\author[Schotland]{John C. Schotland }
\address{Department of Mathematics and Department of Physics, Yale University, New Haven, CT, USA}
\email{john.schotland@yale.edu}

\begin{document}

\begin{abstract}
We study inverse problems for a time-harmonic one-photon model describing the interaction of a single photon with a medium of stationary two-level atoms.
After time-harmonic reduction, the unknown compactly supported atomic density
appears as a frequency-dependent potential in a scattering equation for the half Laplacian. We prove high-frequency uniqueness results for three types of intensity data: source-driven measurements, renormalized far-field
intensity measurements, and phaseless far-field measurements obtained from
coherent superpositions of incident plane waves. In each case, the
corresponding data, given at all sufficiently large frequencies, determine
the atomic density uniquely; the source-driven result requires a geometric
visibility condition on the source and observation sets.
\end{abstract}

\maketitle
\tableofcontents

\section{Introduction}
\subsection{Physical motivation and the model}

This paper is concerned with inverse problems for a quantum model of
light--matter interactions. The physical setting is a single photon propagating
through a medium of stationary two-level atoms, described by an atomic number
density $\rho(x)$ supported in a bounded region of $\mathbb R^n$. The photon
may be absorbed, exciting an atom, and may subsequently be re-emitted. The
resulting cooperative dynamics is a basic many-body process in quantum optics
and is closely related to phenomena such as single-photon superradiance and
subradiance~\cite{Dicke_1954,Gross_1982,Kraisler_Schotland}.

A continuum real-space formulation of this problem was introduced in~\cite{Kraisler_Schotland}. In the one-photon sector, the photon and atomic amplitudes satisfy the coupled system
\begin{equation}\label{eq:dynamical_intro}
    \begin{cases}
        i\partial_t \psi = c(-\Delta)^{1/2}\psi + g\,\rho\, a, \\[2pt]
        i\partial_t a = g\,\psi + \Omega\, a.
    \end{cases}
\end{equation}
Here $\psi(x,t)$ is the photon amplitude, $a(x,t)$ is the atomic excitation
amplitude, $\Omega>0$ is the atomic resonance frequency, $g>0$ is the
atom--field coupling constant, and $c>0$ is the speed of light. The operator
$(-\Delta)^{1/2}$ reflects the linear photon dispersion relation
$\omega_{\mathbf k}=c|\mathbf k|$. The atomic density $\rho$ enters the
system through the coupling term and is the unknown coefficient in the inverse problems considered below. The present work complements our earlier study of a time-dependent inverse problem for entangled two-photon states~\cite{LNOS_two_photon}.

\subsection{Inverse problems and main results}

The purpose of the present paper is to initiate the study of inverse problems
for the time-harmonic form of \eqref{eq:dynamical_intro}. We set
$c=1$, which amounts to a rescaling of time, and look for monochromatic
solutions of the form
\begin{equation*}
    \psi(x,t)=e^{-ikt}\psi(x),
    \qquad
    a(x,t)=e^{-ikt}a(x),
\end{equation*}
where $k>0$ is the frequency. This gives the stationary system
\begin{equation}\label{eq:intro_stationary_system}
    \begin{cases}
        k\psi = (-\Delta)^{1/2}\psi + g\rho a,\\
        ka = g\psi+\Omega a.
    \end{cases}
\end{equation}
For $k\neq\Omega$, the second equation determines $a$ algebraically in
terms of $\psi$, and substitution into the first one yields the equation
\begin{equation}\label{eq:intro_Pk}
    P_k\psi :=
    \left(
        (-\Delta)^{1/2}-k+\frac{g^2}{k-\Omega}\rho
    \right)\psi = 0.
\end{equation}
Thus the unknown atomic density $\rho$ enters as an effective potential in a
fractional equation. We assume throughout that
$\rho\in C_c^\infty(\mathbb R^n)$ is non-negative and not identically zero,
and we write $\Sigma=\operatorname{supp}(\rho)$.

The inverse problem we consider is to determine $\rho$ from measurements of the
photon field $\psi$. A distinctive feature of the present model is that
the physically natural data are of \emph{intensity-type}. Photodetection
gives access to probability densities such as $|\psi|^2$, but the complex
phase of $\psi$ is not directly measurable. The inverse problems we
consider are therefore phaseless: the unknown $\rho$ must be reconstructed
from squared moduli of the photon amplitude.

We study three types of intensity measurements, described informally in the following. The first is a near-field measurement generated by localized sources. Sources are placed outside the
support of $\rho$, and the resulting intensity is observed in another
exterior region. The question is whether these source-to-intensity data
determine the atomic density.

The second and third measurements are taken in the far-field. In the second case, for each illumination the medium is probed by a single incident
plane wave, with the incident direction varying over $\mathbb{S}^{n-1}$, and one measures a suitably renormalized intensity of the total field at infinity. Although the total intensity tends to the background value at large distance,
its next-order asymptotic term contains information about the scattering
amplitude.

In the third case, the incident field is a coherent superposition of two plane
waves with several prescribed relative phases. The corresponding phaseless
far-field measurements allow us to recover certain mixed products of scattering
amplitudes, which are sufficient for uniqueness.

The precise definitions of these three data sets are given in
Section~\ref{sec:formulation}. Our main results, stated there as
Theorems~\ref{thm:source_recovery},
\ref{thm:intensity_determines_rho}, and
\ref{thm:phaseless_superposed}, show that each data set determines
$\rho$ uniquely at all sufficiently large frequencies. The source-driven
result requires a geometric visibility condition on the source and
observation sets; the two far-field results do not.

\subsection{Far-field measurements}
A common feature of the three inverse problems is that the data are
intensities: photodetection measures $|\psi|^2$, and not the phase of
$\psi$ directly. It is therefore natural, in a far-field scattering setting,
to compare the present data with the usual phaseless far-field pattern
$|A_k(\theta,\alpha)|^2$. In that formulation the phase of the scattering
amplitude is lost, and uniqueness typically requires additional structure,
such as reference scatterers, superpositions of incident waves, or
measurements at many frequencies.
The second inverse problem in this paper is based on a different observation.
The measured quantity is the intensity of the total field
\begin{equation*}
    \psi^{\mathrm{in}}_{k,\alpha}
    +
    \psi^{\mathrm{sc}}_{k,\alpha},
\end{equation*}
where the incident plane wave $\psi^{\mathrm{in}}_{k,\alpha}(x)=e^{ik\alpha\cdot x}$ with $\alpha \in \mathbb{S}^{n-1}$ is known. Thus
\begin{equation*}
    \left|
        \psi^{\mathrm{in}}_{k,\alpha}
        +
        \psi^{\mathrm{sc}}_{k,\alpha}
    \right|^2
    =
    1
    +
    2\,\Re\left(
        \overline{\psi^{\mathrm{in}}_{k,\alpha}}\,
        \psi^{\mathrm{sc}}_{k,\alpha}
    \right)
    +
    \left|\psi^{\mathrm{sc}}_{k,\alpha}\right|^2 .
\end{equation*}
The cross term is an interference term between the known incident wave and
the outgoing scattered wave. After subtracting the background intensity and
applying the natural far-field renormalization, which we encode in a
renormalized total-field intensity $\mathcal I_\rho$
(defined in Section~\ref{sec:formulation}), this interference term produces an
oscillatory quantity, oscillating in the radial variable at a frequency fixed
by the incident and observation directions. More precisely, for non-forward
directions $\theta\neq\alpha$, we prove that
\begin{equation*}
    A_k(\theta,\alpha)
    =
    \lim_{R\to\infty}
    \frac{1}{R}
    \int_R^{2R}
    \mathcal I_\rho(\alpha,k,\theta,r)\,
    e^{-ikr(1-\alpha\cdot\theta)}\,dr .
\end{equation*}
Thus, the asymptotic behaviour of the renormalized total-field intensity
determines the full complex scattering amplitude $A_k(\theta,\alpha)$.
In this sense, the second inverse problem is closer to an interferometric
measurement model than to the standard phaseless far-field pattern problem:
the known incident wave serves as a reference field, and no separate
phase-retrieval step is needed once the renormalized asymptotic intensity is
known. The third inverse problem, by contrast, uses genuinely phaseless far-field
data. There the measured quantities are squared moduli
$|A_k^{(a)}|^2$ of far-field amplitudes corresponding to superpositions of
two incident plane waves. Combining the four relative phases
$a\in\{1,-1,i,-i\}$ recovers the mixed products
\begin{equation*}
    A_k(\theta,\alpha_2)\overline{A_k(\theta,\alpha_1)},
\end{equation*}
which are then used, through their high-frequency asymptotics, to determine
$\rho$.

\subsection{Related work}

While inverse problems for the one-photon model~\eqref{eq:dynamical_intro}
have not been considered before (to the best of our knowledge), the half-Laplacian and, more
generally, the fractional Laplacians $(-\Delta)^s$ with $s\in(0,1)$ have
received considerable attention in the inverse problems literature.

In the Calder\'on-type setting, Ghosh, Salo,
and Uhlmann~\cite{Ghosh_Salo_Uhlmann_2020} proved that exterior
Dirichlet-to-Neumann data determine the potential in
\begin{equation*}
    (-\Delta)^s u+q u=0,\qquad 0<s<1.
\end{equation*}
This line of work relies on the unique continuation and Runge approximation
properties of the fractional Laplacian, and has since been developed in many
directions; see, for example,
\cite{Ruland_Salo_2020,Cekic_Lin_Ruland_2020,Covi_2020,
Ghosh_Ruland_Salo_Uhlmann_2021}. The present paper is different both in
setting and in data: we work on all of $\mathbb R^n$ in a scattering
framework, and the data are high-frequency intensity measurements rather
than exterior Dirichlet-to-Neumann data.

A closer scattering-theoretic direction concerns the half Laplacian and its
fractional generalizations. For the time-dependent problem, scattering and
inverse scattering have been studied, for example,
in~\cite{Kitada2010,Ishida2020,IshidaWada2020}; see
also~\cite{Jung1997} for the related Dirac equation. For further recent results on fractional Helmholtz
equations, we refer
to~\cite{CovideHoopSalo2025,ZilberbergCakoniVogelius}. In the
stationary setting, Uhlmann and Wang~\cite{Uhlmann_Wang_2025} recover the
asymptotic behaviour of polyhomogeneous potentials at infinity from
fixed-energy scattering data. In contrast, we recover the
compactly supported coefficient $\rho$, which enters through the
frequency-dependent  potential $V_k = g^2\rho/(k-\Omega)$, from intensity data at all sufficiently large frequencies.

Classical inverse scattering for Schr\"odinger operators provides another
relevant context. For $-\Delta+V$, high-frequency Born approximation methods
show that the leading term of the scattering amplitude determines the
Fourier transform of the potential; see, for instance,
Saito~\cite{Saito_1982}. In phaseless inverse scattering, uniqueness from
phaseless data often requires additional information such as reference
sources, background scatterers, superpositions of incident waves, or
varying-frequency measurements; see, for example,
\cite{Klibanov_2014,Klibanov_Romanov_2016,Novikov_2015,Novikov_2016,
Ivanyshyn_Kress_2010}. The superposition mechanism is particularly close to
our third inverse problem: by using two incident plane waves with four
prescribed relative phases, we recover mixed products of scattering
amplitudes, which suffice to determine $\rho$.

\subsection{Organization of the paper}

The paper is organized as follows. In Section~\ref{sec:formulation} we give
the precise formulation of the three inverse problems and state the main
uniqueness results. Section~\ref{sec:direct} develops the direct theory for
$P_k$, including the uniform free resolvent estimate, the limiting absorption
principle for the perturbed operator, and the far-field expansion. In
Section~\ref{sec:source} we prove the source-driven uniqueness result by
reducing the problem to the recovery of line integrals of $\rho$ and then to
the injectivity of the X-ray transform. Section~\ref{sec:farfield} proves
the two far-field uniqueness results, first for renormalized far-field
intensity data and then for phaseless superposition data. The appendix contains the technical estimates used in the geometric-optics approximation.

\subsection*{Acknowledgements}
G.C.\ and M.S.~were partly supported by Research Council of Finland (Centre of Excellence in Inverse Modelling and Imaging and FAME Flagship, grants 353091 and 359208). M.L.\ was partially supported by the Advanced Grant project 101097198 of the European Research Council, Centre of Excellence of Research Council of Finland (grant 336786) and the FAME flagship of the Research Council of Finland (grant 359186). L.O. and M.N. were supported by the European Research Council of the European Union, grant 101086697 (LoCal),
and the Research Council of Finland, grants 347715,
353096 (Centre of Excellence of Inverse Modelling and Imaging)
and 359182 (Flagship of Advanced Mathematics for Sensing Imaging and Modelling).
M.L. and L.O  were partially supported by Finnish Quantum Flagship. J.C.S was partially supported by Simons Foundation grant MPS-TSRG-00023182-04.
The views and opinions expressed are those of the authors only and do not necessarily reflect those of the funding agencies or the EU.

\section{Formulation of the inverse problems and main results}\label{sec:formulation}

Let $n\geq 2$ and let $\rho\in C_c^\infty(\mathbb{R}^n)$ be a non-negative,  non-zero function with
\begin{equation*}
    \Sigma=\operatorname{supp}(\rho)\subset \mathbb{R}^n.
\end{equation*}
Let $g,\Omega>0$ be fixed constants and let $k>0$ satisfy $k\neq \Omega$.
Eliminating $a$ from the stationary one-photon system
\begin{equation*}
    \begin{cases}
        k\psi = (-\Delta)^{1/2}\psi + g\rho a,\\
        k a = g\psi + \Omega a,
    \end{cases}
\end{equation*}
yields a single equation for $\psi$:
\begin{equation}\label{eq:Pk}
    P_k\psi :=
    \left((-\Delta)^{1/2}-k+\frac{g^2}{k-\Omega}\rho\right)\psi = 0.
\end{equation}
We also use the notation
\begin{equation*}
    P_k = P_k^0 + V_k,
    \qquad
    P_k^0 = (-\Delta)^{1/2}-k,
    \qquad
    V_k = \frac{g^2}{k-\Omega}\rho.
\end{equation*}

In this paper we consider three inverse problems associated with \eqref{eq:Pk}.
The first one is a source-driven near-field problem. The second and third ones
are far-field problems formulated in terms of intensity-type data.

\subsection{Inverse Problem I: Source-driven near-field data}
Let $S\subset \mathbb{R}^n$ be a non-empty open set such that $S\cap \Sigma=\varnothing$.
Given $f\in C_c^\infty(S)$, we consider the forced problem
\begin{equation}
\label{eq:forced}
P_k\psi = f \quad \text{in }\mathbb{R}^n.
\end{equation}
By Lemma~\ref{lem_LAP}, there exists $k_\rho>0$ such that, for each
$k>k_\rho$, problem \eqref{eq:forced} has a unique outgoing solution in
the sense of the limiting absorption principle. We denote this solution
by $\psi^f$.

Let $W\subset \mathbb{R}^n$ be a non-empty open set with $W\cap \Sigma=\varnothing$.
For each $k>k_\rho$, the corresponding measurement map is
\begin{align*}
\Lambda_k^{\mathrm{src}}: C_c^\infty(S) &\to C^0(W),\\
\Lambda_k^{\mathrm{src}}(f) &= \left|\psi^f\big|_{W}\right|^2.
\end{align*}
The first inverse problem is to determine the density $\rho$ from the knowledge of the family of measurement maps $\{\Lambda_k^{\mathrm{src}}\}$ for all sufficiently large $k$.

\begin{remark}\label{rem:disjointness}
The assumptions $S\cap\Sigma=\varnothing$ and $W\cap\Sigma=\varnothing$ are
made only for convenience: since $\Sigma$ is closed, one may replace $S$ and
$W$ by the open sets $S\setminus\Sigma$ and $W\setminus\Sigma$, and all the
results below remain valid provided the latter sets satisfy the conditions
imposed on $S$ and $W$.
\end{remark}

\subsection{Inverse Problem II: Renormalized far-field intensity data}

For $\alpha\in\mathbb{S}^{n-1}$, we consider the incident plane wave
\begin{equation*}
    \psi^{\mathrm{in}}_{k,\alpha}(x)=e^{ik\alpha\cdot x}.
\end{equation*}
The corresponding total field is written as
\begin{equation}\label{eq:total_field}
    \psi_{k,\alpha}
    =
    \psi^{\mathrm{in}}_{k,\alpha}
    +
    \psi^{\mathrm{sc}}_{k,\alpha},
\end{equation}
where the scattered field $\psi^{\mathrm{sc}}_{k,\alpha}$ is the unique outgoing
solution of
\begin{equation}\label{eq:sc_equation}
    P_k\psi^{\mathrm{sc}}_{k,\alpha}
    =
    -V_k\psi^{\mathrm{in}}_{k,\alpha}
    \quad \text{in }\mathbb{R}^n.
\end{equation}
The right-hand side belongs to $C_c^\infty(\mathbb{R}^n)$. By Lemma~\ref{lem_LAP}, for every admissible density $\rho$ there exists
$k_\rho>0$ such that, for all $k>k_\rho$, equation
\eqref{eq:sc_equation} has a unique outgoing solution in the sense of the
limiting absorption principle. Consequently, the total field
$\psi_{k,\alpha}$ satisfies
\begin{equation*}
    P_k\psi_{k,\alpha}=0
    \quad \text{in }\mathbb{R}^n.
\end{equation*}

As shown later in Lemma~\ref{lem:farfield_expansion}, the total field admits
the far-field expansion
\begin{equation}\label{eq:farfield}
    \psi_{k,\alpha}(r\theta)
    =
    e^{ik\alpha\cdot(r\theta)}
    +
    r^{-(n-1)/2}e^{ikr}A_k(\theta,\alpha)
    +
    o(r^{-(n-1)/2}),
    \qquad r\to\infty,
\end{equation}
where $\theta\in\mathbb{S}^{n-1}$. The coefficient $A_k(\theta,\alpha)$ is the scattering amplitude.

In the far-field regime, the directly measurable quantity is the intensity $|\psi_{k,\alpha}(r\theta)|^2$. Using \eqref{eq:farfield}, we obtain
\begin{equation}\label{eq:intensity_expansion_justification}
    |\psi_{k,\alpha}(r\theta)|^2
=
1
+
2r^{-(n-1)/2}
\Re\Bigl(
e^{ikr(1-\alpha\cdot\theta)}A_k(\theta,\alpha)
\Bigr)
+
o\bigl(r^{-(n-1)/2}\bigr),
\qquad r\to\infty.
\end{equation}
In particular, $|\psi_{k,\alpha}(r\theta)|^2 \to 1$ as $r\to\infty$, so the limit itself carries no information about $\rho$. The relevant object is the asymptotic behaviour at the next order, namely
\begin{equation*}
    2r^{-(n-1)/2}
    \Re\Bigl(
        e^{ikr(1-\alpha\cdot\theta)}A_k(\theta,\alpha)
    \Bigr).
\end{equation*}
Accordingly, we introduce the renormalized intensity
\begin{equation}\label{eq:renormalized_intensity}
    \mathcal I_\rho(\alpha,k,\theta,r)
    :=
   r^{(n-1)/2}
    \Bigl(
        |\psi_{k,\alpha}(r\theta)|^2-1
    \Bigr),
\end{equation}
which by \eqref{eq:intensity_expansion_justification} satisfies
\begin{equation}\label{eq:Irho_asymptotics}
    \mathcal I_\rho(\alpha,k,\theta,r)
    =
    2\Re\Bigl(
        e^{ikr(1-\alpha\cdot\theta)}A_k(\theta,\alpha)
    \Bigr)
    +
    o(1),
    \qquad r\to\infty.
\end{equation}

The data in the second inverse problem are not the values of
$\mathcal I_\rho$ at finite radii. Rather, they are its asymptotic behaviour as
$r\to\infty$, modulo terms that vanish at infinity.

We say that two renormalized far-field intensities $\mathcal I_{\rho_1}$ and
$\mathcal I_{\rho_2}$ are equivalent at infinity, and write
\begin{equation*}
\mathcal I_{\rho_1}\sim_\infty \mathcal I_{\rho_2},
\end{equation*}
if there exists $K>0$ such that
\begin{equation}\label{eq:equivalence_relation}
\mathcal I_{\rho_1}(\alpha,k,\theta,r)
-
\mathcal I_{\rho_2}(\alpha,k,\theta,r)
\to 0
\qquad \text{as } r\to\infty
\end{equation}
for every $k>K$ and every $\alpha,\theta\in\mathbb{S}^{n-1}$ with
$\theta\neq\alpha$. We denote the corresponding asymptotic equivalence class by $[\mathcal I_\rho]_\infty$.

The second inverse problem is to determine the density $\rho$ from the knowledge of the asymptotic equivalence
class $[\mathcal I_\rho]_\infty$ for all sufficiently large $k$.

\subsection{Inverse Problem III: Phaseless far-field data from superposed waves}

In the third inverse problem, we consider phaseless far-field measurements arising from coherent superpositions of two plane waves.

Let $k>k_\rho$, $\alpha_1,\alpha_2,\theta\in\mathbb{S}^{n-1}$, and $a\in\mathbb{C}$. We consider the incident field
\begin{equation*}
    \psi^{\mathrm{in},a}_{k,\alpha_1,\alpha_2}(x)
    =
    e^{ik\alpha_1\cdot x}
    +
    a\,e^{ik\alpha_2\cdot x}.
\end{equation*}
Let $\psi^a_{k,\alpha_1,\alpha_2}$ denote the corresponding total field, and let $A_k^{(a)}(\theta,\alpha_1,\alpha_2)$ denote the corresponding scattering amplitude defined via the analogue of \eqref{eq:farfield}. Since the scattering problem is linear in the incident field,
\begin{equation}\label{eq:superposition_amplitude}
    A_k^{(a)}(\theta,\alpha_1,\alpha_2)
    =
    A_k(\theta,\alpha_1)
    +
    a\,A_k(\theta,\alpha_2).
\end{equation}

The measurement maps are
\begin{align*}
    \mathcal M_k:
    \mathbb S^{n-1}\times\mathbb S^{n-1}
    \times \{1,-1,i,-i\}
    \times \mathbb S^{n-1}
    &\to \mathbb R,\\
    \mathcal M_k(\alpha_1,\alpha_2,a,\theta)
    &=
    \left|A_k^{(a)}(\theta,\alpha_1,\alpha_2)\right|^2.
\end{align*}

The third inverse problem is to determine the density $\rho$ from the knowledge
of the measurement maps $\mathcal M_k$ for all sufficiently large $k$.

\subsection{Main results}
We now state the main uniqueness results proved in this paper. The first result
concerns the source-driven near-field data. It requires the following geometric
condition on the source set, the observation set, and the support of the density.

For $z,p\in\mathbb{R}^n$, we denote by $[z,p]$ the closed line segment
joining $z$ to $p$. For $z\in\mathbb{R}^n$ and a set
$A\subset\mathbb{R}^n$, we define
\begin{equation*}
    (z;A] := \bigcup_{p\in A}[z,p]\setminus\{z\}.
\end{equation*}

\begin{condition}\label{cond:geometric}
There exist points $x,y\in S$ and a non-empty open set $W'\subset W$ such that
\begin{equation*}
    \Sigma \subset (x;W']
    \qquad
    \text{and}
    \qquad
    \Sigma \cap (y;W'] = \varnothing.
\end{equation*}
\end{condition}

Geometrically, $W'$ acts as a screen behind $\Sigma$ as viewed from $x$, while
from the reference source $y$ the same screen is reached without crossing
$\Sigma$.

\begin{theorem}[Determination from source-to-solution data]
\label{thm:source_recovery}
Assume that $S$ is open, $\Sigma\cap(S\cup W)=\varnothing$, and that
Condition~\ref{cond:geometric} holds. Then the source-to-solution intensity
data $\{\Lambda_k^{\mathrm{src}}\}$ for all sufficiently large $k$ uniquely determine $\rho$.
\end{theorem}

The second result concerns the renormalized far-field intensity data. Here the
data are the asymptotic equivalence classes $[\mathcal I_\rho]_\infty$
introduced above.

\begin{theorem}[Determination from renormalized far-field intensity data]
\label{thm:intensity_determines_rho}
The asymptotic equivalence class $[\mathcal I_\rho]_\infty$ uniquely
determines $\rho$.
\end{theorem}

Finally, we consider the phaseless far-field data obtained from coherent
superpositions of two incident plane waves.

\begin{theorem}[Determination from phaseless superposition data]
\label{thm:phaseless_superposed}
The phaseless superposition data $\{\mathcal M_k\}$ for all sufficiently large $k$ uniquely determine $\rho$.
\end{theorem}

\section{Direct problem}\label{sec:direct}

In this section, we study the well-posedness of the direct problem associated with the operator $P_k$ and establish the corresponding limiting absorption principle.

We work with weighted Sobolev spaces. For $l, \delta \in \mathbb{R}$, 
the space $H^l_\delta(\mathbb{R}^n)$ is defined as the completion of 
$C_c^\infty(\mathbb{R}^n)$ with respect to the norm
\begin{equation*}
    \|u\|_{H^l_\delta(\mathbb{R}^n)}
    =
    \left(
        \int_{\mathbb{R}^n}
        \langle x \rangle^{2\delta}
        |(1-\Delta)^{l/2} u(x)|^2
         dx
    \right)^{1/2}.
\end{equation*}
Here $\langle x \rangle = (1+|x|^2)^{1/2}$, and the operator $(1-\Delta)^{l/2}$ is defined via the Fourier transform by
\begin{equation*}
    \mathcal{F}[(1-\Delta)^{l/2} u](\xi)
    =
    (1+|\xi|^2)^{l/2} \hat u(\xi).
\end{equation*}
We define $L^2_\delta(\mathbb{R}^n) := H^0_\delta(\mathbb{R}^n)$.

\begin{remark}\label{notations}
Throughout the paper the dimension $n\ge 2$, the atomic resonance frequency
$\Omega>0$, and the atom-field coupling constant $g>0$ are fixed.
Accordingly, the dependence of constants on $n$, $\Omega$, and $g$ is
suppressed unless explicitly indicated otherwise, and we write $C$ for a
positive constant that may change from line to line. Dependence on all other
parameters is retained explicitly. For $x\in\mathbb{R}$, we denote by $[x]$
the integer part of $x$.
\end{remark}

\subsection{Free fractional resolvent estimate}
We first establish a uniform resolvent bound for
$$
P_k^0 = (-\Delta)^{1/2} - k
$$
by reducing the problem to classical high-frequency resolvent estimates for the Laplacian.

Recall the standard limiting absorption estimate for the Laplacian in
$\mathbb{R}^n$ (see \cite{Agmon1975}, \cite[Section~7.1]{Yafaev}):
if $\delta > 1/2$ and $\lambda \geq 1$, then the limit
\begin{equation*}
    (-\Delta - \lambda \pm i0)^{-1}
    :=
    \lim_{\varepsilon\to 0^+} (-\Delta - \lambda \pm i\varepsilon)^{-1}
\end{equation*}
is well defined as a bounded operator from $L^2_\delta(\mathbb{R}^n)$ to
$L^2_{-\delta}(\mathbb{R}^n)$, and for any multi-index $\alpha$ with
$|\alpha| \leq 1$,
\begin{equation}\label{eq:laplacian_resolvent}
\|\partial^\alpha (-\Delta - \lambda \pm i0)^{-1} f\|_{L^2_{-\delta}}
\leq C \lambda^{\frac{|\alpha|-1}{2}} \|f\|_{L^2_\delta},
\end{equation}
where $C$ is independent of $\lambda \geq 1$ and
$f\in L^2_\delta(\mathbb{R}^n)$.

On the other hand, by the limiting absorption principle for the free
fractional Laplacian established in \cite{Ben-Artzi_Nemirovsky}, for every fixed $k>0$ the boundary values
\begin{equation}\label{frac_L_epsilon_eq_0}
    ((-\Delta)^{1/2}-k\pm i0)^{-1}
    :=
    \lim_{\varepsilon\to0^+}
    ((-\Delta)^{1/2}-k\pm i\varepsilon)^{-1}
\end{equation}
exist as bounded operators from $L^2_\delta(\mathbb{R}^n)$ to
$L^2_{-\delta}(\mathbb{R}^n)$, provided $\delta>1/2$.
The next lemma shows that the corresponding operator norms can be chosen uniformly for $k\geq 1$.

\begin{lemma}\label{lem:free_frac}
Let $n \geq 2$ and $\delta>\frac{1}{2}$. Then there exists $C_{\delta}>0$
such that for all $k\geq 1$, $\varepsilon\geq 0$, and
$f\in L^2_\delta(\mathbb R^n)$,
\begin{equation}\label{eq:frac_est}
\|((-\Delta)^{1/2} - k - i\varepsilon)^{-1}f\|_{L^2_{-\delta}} \leq C_\delta\|f\|_{L^2_\delta}.
\end{equation}
For $\varepsilon=0$, the operator in \eqref{eq:frac_est} is understood as the
outgoing boundary value defined in \eqref{frac_L_epsilon_eq_0}.
\end{lemma}
\begin{proof}
We first prove the estimate for $\varepsilon>0$. Set $w=k+i\varepsilon$, and
note that $|t-w|\geq\varepsilon$ and $|t+w|\geq k$ for $t\geq 0$, while
$\Im w^2=2k\varepsilon>0$, so that $w^2$ lies outside the spectrum
$[0,\infty)$ of $-\Delta$. As Fourier multipliers, the elementary identities
\[
    \frac{1}{t-w}
    =
    \frac{t+w}{t^2-w^2}
    =
    \frac{1}{t+w}+\frac{2w}{t^2-w^2},
    \qquad t\geq 0,
\]
give
\begin{equation}\label{eq:factorization_eps}
    \left((-\Delta)^{1/2}-k-i\varepsilon\right)^{-1}
    =
    \bigl((-\Delta)^{1/2}+w\bigr)
    \bigl(-\Delta-w^2\bigr)^{-1}
\end{equation}
and
\begin{equation}\label{eq:partial_fractions}
    \left((-\Delta)^{1/2}-k-i\varepsilon\right)^{-1}
    =
    \bigl((-\Delta)^{1/2}+w\bigr)^{-1}
    +
    2w\bigl(-\Delta-w^2\bigr)^{-1}.
\end{equation}
The factorization \eqref{eq:factorization_eps} will be needed later; here we
use \eqref{eq:partial_fractions}. The multiplier of the first operator in
\eqref{eq:partial_fractions} is bounded by $1/k\leq 1$, so
\[
    \|((-\Delta)^{1/2}+w)^{-1}f\|_{L^2_{-\delta}}
    \leq
    \|((-\Delta)^{1/2}+w)^{-1}f\|_{L^2}
    \leq
    \|f\|_{L^2}
    \leq
    \|f\|_{L^2_\delta}.
\]
For the second operator, since $w^2\in\mathbb{C}\setminus[0,\infty)$ and
$|w^2|=k^2+\varepsilon^2\geq 1$, the Agmon estimate for complex spectral
parameter (see \cite[Theorem~7.1.1]{FeldmanSaloUhlmann}, whose constant
depends only on $n$, $\delta$, and a lower bound for the modulus of the
spectral parameter, here equal to $1$) gives
\[
    2|w|\,\|(-\Delta-w^2)^{-1}f\|_{L^2_{-\delta}}
    \leq
    2C_\delta\,|w|\,|w^2|^{-1/2}\,\|f\|_{L^2_\delta}
    =
    2C_\delta\,\|f\|_{L^2_\delta}.
\]
Combining the two bounds with \eqref{eq:partial_fractions} proves
\eqref{eq:frac_est} for $\varepsilon>0$.

Let $f\in L^2_\delta(\mathbb R^n)$. By the triangle inequality,
\begin{multline*}
   \|\left((-\Delta)^{1/2}-k-i0\right)^{-1}f\|_{L^2_{-\delta}}
    \leq \|\left((-\Delta)^{1/2}-k-i\varepsilon\right)^{-1}f\|_{L^2_{-\delta}}\\
    + \|\left((-\Delta)^{1/2}-k-i0\right)^{-1}f - \left((-\Delta)^{1/2}-k-i\varepsilon \right)^{-1}f\|_{L^2_{-\delta}}.
\end{multline*}
By \eqref{eq:frac_est} for $\varepsilon>0$, the first term on the right hand
side is bounded from above by $C\|f\|_{L^2_\delta}$. Then, by taking the
limit $\varepsilon\to 0$ and using \eqref{frac_L_epsilon_eq_0}, we obtain
\eqref{eq:frac_est} for $\varepsilon=0$.
\end{proof}

For $\varepsilon>0$ we set 
\begin{equation*}
    R_k^\varepsilon := ((-\Delta)^{1/2} - k - i\varepsilon)^{-1} = (P_k^0 - i\varepsilon)^{-1},
    \qquad
    R_k^0 := ((-\Delta)^{1/2} - k - i0)^{-1}.
\end{equation*}

As an immediate consequence, we obtain the corresponding estimate in weighted Sobolev spaces.

\begin{corollary}\label{cor:free_frac_Hl}
Let $n\geq 2$, let $\delta> 1/2$, and let $l\geq 0$ be an integer.
Then there exists $C=C_{\delta,l}>0$ such that, for all
$k\geq 1$, $\varepsilon\geq 0$, and
$f\in H^l_\delta(\mathbb R^n)$,
\begin{equation}\label{eq:frac_Hl_est}
    \|R_k^\varepsilon f\|_{H^l_{-\delta}}
    \leq
    C_{\delta,l}\|f\|_{H^l_\delta}.
\end{equation}
For $\varepsilon=0$, the operator $R_k^\varepsilon$ in
\eqref{eq:frac_Hl_est} is understood as the outgoing boundary value $R_k^0$.
Moreover, for every fixed $k\geq 1$ and every
$f\in H^l_\delta(\mathbb R^n)$,
\begin{equation}\label{eq:free_limit_Hl}
    R_k^\varepsilon f \to R_k^0 f
    \quad \text{in } H^l_{-\delta}(\mathbb R^n)
    \quad \text{as } \varepsilon\to0^+.
\end{equation}
\end{corollary}

\begin{proof}
We first consider $\varepsilon >0$. Since $R_k^\varepsilon$ and
$(1-\Delta)^{l/2}$ are Fourier multipliers, they commute. Therefore, by Lemma~\ref{lem:free_frac},
\begin{equation*}
    \|R_k^\varepsilon f\|_{H^l_{-\delta}}
    =
    \|R_k^\varepsilon (1-\Delta)^{l/2}f\|_{L^2_{-\delta}} 
    \leq
    C\|(1-\Delta)^{l/2}f\|_{L^2_\delta}  
    =
    C\|f\|_{H^l_\delta}.
\end{equation*}
This proves \eqref{eq:frac_Hl_est} for $\varepsilon >0$.

For $\varepsilon=0$, the same triangle-inequality argument as in the proof of
Lemma~\ref{lem:free_frac}, together with \eqref{frac_L_epsilon_eq_0}, gives
the estimate for the outgoing boundary value. Thus \eqref{eq:frac_Hl_est}
also holds for $\varepsilon=0$.

It remains to prove \eqref{eq:free_limit_Hl}. Since $R_k^\varepsilon$ and
$(1-\Delta)^{l/2}$ commute,
\[
    \|R_k^\varepsilon f-R_k^0 f\|_{H^l_{-\delta}}
    =
    \|R_k^\varepsilon (1-\Delta)^{l/2}f
      -R_k^0 (1-\Delta)^{l/2}f\|_{L^2_{-\delta}}.
\]
Since $(1-\Delta)^{l/2}f\in L^2_\delta(\mathbb R^n)$, the convergence
\eqref{frac_L_epsilon_eq_0} implies that the right-hand side tends to zero as
$\varepsilon\to0^+$. This proves \eqref{eq:free_limit_Hl}.
\end{proof}

\begin{corollary}\label{cor:r_smooth}
    If $f\in C_c^\infty(\mathbb{R}^n)$, then $R_k^\varepsilon f \in C^\infty(\mathbb{R}^n)$ for every
$k\geq 1$ and $\varepsilon \geq 0$.
\end{corollary}
\begin{proof}
    Follows from Corollary \ref{cor:free_frac_Hl}.
\end{proof}

\subsection{Limiting absorption principle for the perturbed operator}
We now establish the limiting absorption principle for the perturbed operator $P_k$.
The argument is based on the resolvent identity and a Neumann series expansion, using the uniform bounds for the free resolvent obtained above.
\begin{lemma}\label{lem_LAP}
Let $n \geq 2$. Fix $\delta >1/2$ and an integer $l \geq 0$. There exists $k_\rho = k_{\rho, l, \delta}>0$ and a constant $C_{ \rho,l, \delta}>0$ such that for every $k > k_\rho$ and every $f \in H^l_\delta(\mathbb{R}^n)$, the limit
\begin{equation}\label{eq:LAP_def}
\psi^f := \lim_{\varepsilon \to 0^+}(P_k - i\varepsilon)^{-1}f
\end{equation}
exists in $H^l_{-\delta}(\mathbb{R}^n)$, and 
\begin{equation}\label{eq:weighted_stability}
\|\psi^f\|_{H^l_{-\delta}} \leq C_{ \rho,l, \delta}\|f\|_{H^l_\delta}.
\end{equation}
Moreover, $\psi^f$ admits the Neumann series representation
\begin{equation}\label{eq:psi_equals_Nseries}
\psi^f
=
\sum_{j=0}^\infty
(-R_k^0 V_k)^j R_k^0 f
\quad\text{in } H^l_{-\delta}(\mathbb{R}^n).
\end{equation}

\end{lemma}

\begin{proof}
    For $\varepsilon >0$, set $\psi_\varepsilon = (P_k - i\varepsilon)^{-1}f$. Applying $R_k^\varepsilon$ to the identity $(P_k^0 - i\varepsilon)\psi_\varepsilon = f - V_k\psi_\varepsilon$ gives
    \begin{equation*}
        \psi_\varepsilon = R_k^\varepsilon f - R_k^\varepsilon(V_k\psi_\varepsilon).
    \end{equation*}
    Next, we estimate the potential term. Recall that
    \begin{equation*}
        V_k = \frac{g^2}{k - \Omega}\rho,
    \end{equation*}
    where $\rho \in C_c^\infty(\mathbb{R}^n)$. Note that multiplication by $\rho$ is a bounded operator from $H^l_{-\delta}(\mathbb{R}^n)$ to $H^l_\delta(\mathbb{R}^n)$. Indeed, $\rho$ is smooth with compact support, so it preserves the Sobolev order, and the weight $\langle x\rangle^{\pm\delta}$ is bounded on the support of $\rho$. Consequently, for $k>\Omega$,
    \begin{equation}\label{eq:Vk_bound}
    \|V_k u\|_{H^l_\delta}
    \leq
    C_{\rho,l,\delta} \frac{g^2 }{k - \Omega}
    \|u\|_{H^l_{-\delta}}.
    \end{equation}
    Hence, by Corollary \ref{cor:free_frac_Hl}, for any $u\in H^l_{-\delta}(\mathbb{R}^n)$
    \begin{equation*}
         \|R_k^\varepsilon V_k u \|_{H^l_{-\delta}}\leq C_{l,\delta} \|V_k u\|_{H^l_\delta} \leq C_{\rho,l,\delta} \frac{g^2 }{k - \Omega}
    \|u\|_{H^l_{-\delta}},
    \end{equation*}
    uniformly in $\varepsilon \geq 0$. Therefore, there exists $k_\rho = k_{\rho,l,\delta} >0$ such that for any $k>k_{\rho}$ the following estimates hold
    \begin{equation}\label{eq:contraction}
        \|R_k^\varepsilon V_k\|_{H^l_{-\delta} \to H^l_{-\delta}}
        \leq
        \frac{C_{\rho,l,\delta} }{k} < 1,
    \end{equation}
    uniformly in $\varepsilon\geq 0$. Therefore, the Neumann series
    \begin{equation}\label{eq:Neumann}
    \psi_\varepsilon
    =
    \sum_{j=0}^\infty (-R_k^\varepsilon V_k)^j R_k^\varepsilon f
    \end{equation}
    converges in $H^l_{-\delta}(\mathbb{R}^n)$ for all $\varepsilon>0$ and sufficiently large $k>k_\rho$. Since \eqref{eq:contraction} holds also for $\varepsilon=0$, we can set 
    \begin{equation*}
    \psi_0
    =
    \sum_{j=0}^\infty (-R_k^0 V_k)^j R_k^0 f.
    \end{equation*}
    Using \eqref{eq:contraction}, we estimate
    \begin{equation}\label{eq:uniform_est}
        \|\psi_\varepsilon\|_{H^l_{-\delta}}
        \leq
        \frac{C_{\delta,l}}{1 - C_{\rho,l,\delta}/k}
        \|f\|_{H^l_\delta}
        \leq
        C_{\rho,l,\delta} \|f\|_{H^l_\delta},
    \end{equation}
    where $C_{\rho,l,\delta}>0$ is independent of $k > k_\rho$ and $\varepsilon\geq 0$.

    It remains to show that $\psi_\varepsilon$ converges to $\psi_0$ in $H^l_{-\delta}(\mathbb{R}^n)$ as $\varepsilon \to 0^+$. Let $\tau > 0$. Since \eqref{eq:contraction} holds uniformly in $\varepsilon \geq 0$, the Neumann series converges in operator norm uniformly with respect to $\varepsilon$. Hence we may choose $N$ such that
    \begin{equation}\label{eq:tail}
        \sum_{j > N} \|(-R_k^\varepsilon V_k)^j R_k^\varepsilon f\|_{H^l_{-\delta}}
        < \frac{\tau}{3}
    \end{equation}
    uniformly for $\varepsilon\geq 0$. Therefore, for $\varepsilon  >0$, using the Neumann series
    representation we obtain
    \begin{equation}\label{psiepsilon_minus_psi0}
        \|\psi_\varepsilon - \psi_0\|_{H^l_{-\delta}}
        \leq
        \sum_{j=0}^N
        \left\| (-R_k^\varepsilon V_k)^j R_k^\varepsilon f
        -
        (-R_k^0 V_k)^j R_k^0 f \right\|_{H^l_{-\delta}} + \frac{2\tau}{3}.
    \end{equation}
    We show, by induction, that for each fixed $j \geq 0$,
    \begin{equation}\label{ind}
        (-R_k^\varepsilon V_k)^j R_k^\varepsilon f
        \to
        (-R_k^0 V_k)^j R_k^0 f
        \quad \text{in } H^l_{-\delta}(\mathbb{R}^n)
        \quad \text{as } \varepsilon \to 0^+.
    \end{equation}
    The case $j=0$ follows from Corollary~\ref{cor:free_frac_Hl}. Assume that the claim holds for $j-1$ and set
    \begin{equation*}
        u_0
        =
        (-R_k^0 V_k)^{j-1} R_k^0 f,
        \qquad
        u_\varepsilon
        =
        (-R_k^\varepsilon V_k)^{j-1} R_k^\varepsilon f.
    \end{equation*}
    By the induction hypothesis, $u_\varepsilon \to u_0$ in $H^l_{-\delta}(\mathbb{R}^n)$ as $\varepsilon\to 0^+$. Next, we estimate
    \begin{multline*}
        \left\| (-R_k^\varepsilon V_k)^j R_k^\varepsilon f
        -
        (-R_k^0 V_k)^j R_k^0 f \right\|_{H^l_{-\delta}} = \left\| [R_k^\varepsilon V_k] u_\varepsilon
        -
        [R_k^0 V_k]u_0 \right\|_{H^l_{-\delta}}\\
        \leq \|R_k^\varepsilon V_k\|_{H^l_{-\delta} \to H^l_{-\delta}} \|u_\varepsilon - u_0\|_{H^l_{-\delta}} + \|[R_k^\varepsilon - R_k^0](V_k u_0)\|_{H^l_{-\delta}}
    \end{multline*}
    By \eqref{eq:contraction} and the induction hypothesis, the first term tends to zero as $\varepsilon\to 0^+$. By Corollary \ref{cor:free_frac_Hl}, the second term vanishes as $\varepsilon\to 0^+$. Therefore, \eqref{ind} holds, and consequently, for sufficiently small $\varepsilon>0$, the first term on the right-hand side of \eqref{psiepsilon_minus_psi0} is bounded from above by $\tau/3$. This means that $\psi_\varepsilon$ converges to $\psi_0$ in
    $H^l_{-\delta}(\mathbb{R}^n)$ as $\varepsilon \to 0^+$. Hence the limit in \eqref{eq:LAP_def} exists and
    \begin{equation*}
        \psi^f=\psi_0.
    \end{equation*}
    By the definition of $\psi_0$, we therefore obtain
    \begin{equation*}
        \psi^f
        =
        \sum_{j=0}^\infty
        (-R_k^0 V_k)^j R_k^0 f
        \quad\text{in } H^l_{-\delta}(\mathbb{R}^n),
    \end{equation*}
    which is precisely \eqref{eq:psi_equals_Nseries}. Moreover, since
    \eqref{eq:uniform_est} holds for $\varepsilon=0$, we get
    \eqref{eq:weighted_stability}. This completes the proof.
\end{proof}

\subsection{Far-field asymptotics and the scattering amplitude}

We begin with the far-field expansion of the outgoing Laplacian resolvent applied to compactly supported data. It follows by integrating the asymptotic expansion of the outgoing Green's function of the Helmholtz operator over the support of $h$; see \cite[Eq. (7.3.2)]{FeldmanSaloUhlmann}.

\begin{lemma}\label{lem:laplacian_farfield}
Let $n \geq 2$, $k > 0$, and $h \in L^2(\mathbb{R}^n)$ with compact support.
Then
\begin{equation}\label{eq:laplacian_farfield}
(-\Delta - k^2 - i0)^{-1}h\,(r\theta)
=
d_n\, k^{(n-3)/2}\,
r^{-(n-1)/2}\, e^{ikr}\,\hat{h}(k\theta)
+ O(r^{-(n+1)/2}),
\end{equation}
as $r\to\infty$, where $\theta \in \mathbb{S}^{n-1}$ and $d_n\in \mathbb{C}$ is a non-zero constant depending only on $n$.
The remainder is uniform for $\theta\in\mathbb S^{n-1}$, and the implicit
constant may depend on $n$, $k$, and $h$.
\end{lemma}

We now derive the far-field expansion and amplitude formula for the
fractional operator.

\begin{lemma}\label{lem:farfield_expansion}
Let $n \geq 2$ and $k_\rho$ be the threshold from Lemma \ref{lem_LAP} with $l=[n/2]+1$ and $\delta=3/4$. Then, for $k > k_\rho$ and
$\alpha \in \mathbb{S}^{n-1}$,
the total field $\psi_{k,\alpha}$ defined by \eqref{eq:total_field}
satisfies
\begin{equation}\label{eq:farfield_lemma}
\psi_{k,\alpha}(r\theta)
=
e^{ik\alpha \cdot (r\theta)}
+
r^{-(n-1)/2}\, e^{ikr}\, A_k(\theta,\alpha)
+
o(r^{-(n-1)/2}),
\qquad r\to\infty,
\end{equation}
where the scattering amplitude is given by
\begin{equation}\label{eq:amplitude}
A_k(\theta,\alpha)
=
-2d_n k^{(n-1)/2}
\int_{\mathbb{R}^n}
e^{-ik\theta\cdot y}\,
V_k(y)\,
\psi_{k,\alpha}(y)\,dy.
\end{equation}
\end{lemma}

\begin{proof}
By \eqref{eq:sc_equation} and \eqref{eq:total_field}, the scattered field
satisfies
\begin{equation*}
    P_k^0\,\psi^{\mathrm{sc}}_{k,\alpha}
    =
    -V_k\psi^{\mathrm{in}}_{k,\alpha}-V_k\psi^{\mathrm{sc}}_{k,\alpha}
    =
    -V_k\,\psi_{k,\alpha}.
\end{equation*}
Since $\psi^{\mathrm{sc}}_{k,\alpha}$ is outgoing, we have
\begin{equation*}
    \psi^{\mathrm{sc}}_{k,\alpha}
    =
    -R_k^0 h,
\end{equation*}
where $h = V_k\,\psi_{k,\alpha}$ is a compactly supported function.

By letting $\varepsilon\to0^+$ in \eqref{eq:factorization_eps}, in the sense
of outgoing boundary values, and using \eqref{frac_L_epsilon_eq_0} together
with \eqref{eq:laplacian_resolvent}, we obtain
\begin{equation*}
    R_k^0 = ((-\Delta)^{1/2}+k)\,(-\Delta - k^2 - i0)^{-1}.
\end{equation*}
Set 
\begin{equation*}
    u = (-\Delta - k^2 - i0)^{-1}h,
\end{equation*}
so that
\begin{equation}\label{eq:sc_factor}
\psi^{\mathrm{sc}}_{k,\alpha} = -((-\Delta)^{1/2}+k)\,u
=
 - 2ku - ((-\Delta)^{1/2}-k)u.
\end{equation}
The first term has known asymptotics from Lemma~\ref{lem:laplacian_farfield}:
\begin{equation}\label{eq:u_farfield}
u(r\theta) = d_n\, k^{(n-3)/2}\, r^{-(n-1)/2}\, e^{ikr}\,\hat{h}(k\theta)
+ O(r^{-(n+1)/2}).
\end{equation}
It remains to determine the far-field behaviour of $w:=((-\Delta)^{1/2}-k)u$. Since 
\begin{equation*}
    ((-\Delta) - k^2)u = h,
\end{equation*}
we obtain 
\begin{equation*}
    w=((-\Delta)^{1/2}+k)^{-1}h.
\end{equation*}
Since $-(-\Delta)^{1/2}$ is a non-positive operator, Corollary 3.6 and Theorem 1.10 in \cite{EngelNagel} imply that 
\begin{equation*}
    w
    =
    \int_0^\infty e^{-kt} e^{-t(-\Delta)^{1/2}}h\,dt.
\end{equation*}
Now the semigroup $e^{-t(-\Delta)^{1/2}}$ is given by convolution with the Poisson kernel,
\begin{equation*}
    e^{-t(-\Delta)^{1/2}}h(x)
    =
    c_n\int_{\mathbb{R}^n}
    \frac{t}{(t^2+|x-y|^2)^{(n+1)/2}}\,h(y)\,dy,
    \qquad
    c_n=\frac{\Gamma((n+1)/2)}{\pi^{(n+1)/2}},
\end{equation*}
see \cite[Chapter III, \S2, Proposition 5]{Stein1970} for the Poisson kernel and \cite[Proposition 16.4]{Garofalo2019} for its identification with the semigroup $e^{-t(-\Delta)^{1/2}}$. Therefore,
\begin{equation*}
    w(x)
    =
    c_n\int_0^\infty e^{-kt}
    \int_{\mathbb{R}^n}
    \frac{t}{(t^2+|x-y|^2)^{(n+1)/2}}\,h(y)\,dy\,dt.
\end{equation*}

Since $h\in L_c^2(\mathbb{R}^n)$, its support is contained in some ball $B_R$.
If $|x|>2R$ and $y\in\operatorname{supp}(h)$, then
\begin{equation*}
    |x-y|\ge |x|-|y|\ge \frac{|x|}{2}.
\end{equation*}
Hence
\begin{equation*}
    \frac{t}{(t^2+|x-y|^2)^{(n+1)/2}}
    \le
    \frac{t}{|x-y|^{n+1}}
    \le
    \frac{C\,t}{|x|^{n+1}}.
\end{equation*}
It follows that
\begin{equation*}
    |w(x)|
    \le
    \frac{C}{|x|^{n+1}}
    \int_0^\infty e^{-kt}t\,dt
    \int_{\mathbb{R}^n}|h(y)|\,dy
    \le
    \frac{C_k}{|x|^{n+1}}.
\end{equation*}
Thus
\begin{equation*}
    ((-\Delta)^{1/2}-k)u(x)=O(|x|^{-(n+1)}),
    \qquad |x|\to\infty.
\end{equation*}
Substituting into \eqref{eq:sc_factor} and using \eqref{eq:u_farfield}:
\begin{equation*}
\psi^{\mathrm{sc}}_{k,\alpha} (r\theta)
= - 2k\, d_n\, k^{(n-3)/2}\, r^{-(n-1)/2}\, e^{ikr}\, \hat{h}(k\theta)
+ O(r^{-(n+1)/2}).
\end{equation*}
Recalling that $h=V_k\psi_{k,\alpha}$, we obtain \eqref{eq:amplitude}.
\end{proof}

\section{The source-driven inverse problem}\label{sec:source}

\subsection{Geometric optics approximation}

Let $\delta>\frac12$ and let $l\geq0$ be an integer. For $k\geq1$, $\varepsilon\geq0$, $\alpha\in\mathbb{S}^{n-1}$, and
$h\in H^l_\delta(\mathbb{R}^n)$, we define
\begin{equation}\label{eq:conj_identity}
    \widetilde{R}_{k,\varepsilon}h
    :=
    e^{-ik\alpha\cdot(\cdot)}
    R_k^\varepsilon
    \bigl[e^{ik\alpha\cdot(\cdot)}h\bigr].
\end{equation}
We set $\widetilde{R}_k:=\widetilde{R}_{k,0}$. For fixed $k$ and
$\alpha$, multiplication by $e^{\pm ik\alpha\cdot x}$ is bounded on the
corresponding weighted Sobolev spaces. Hence, by
Corollary~\ref{cor:free_frac_Hl},
$\widetilde{R}_{k,\varepsilon}$ is a bounded operator from
$H^l_\delta(\mathbb{R}^n)$ to $H^l_{-\delta}(\mathbb{R}^n)$.
Moreover, Corollary~\ref{cor:free_frac_Hl} implies
\begin{equation}\label{eq:amp_res_convergence}
    \widetilde{R}_{k,\varepsilon}h
    \longrightarrow
    \widetilde{R}_k h
    \quad\text{in }H^l_{-\delta}(\mathbb{R}^n)
    \quad\text{as }\varepsilon\to0^+.
\end{equation}
In particular, by Corollary~\ref{cor:r_smooth},
\begin{equation}\label{eq:r_smooth}
    \widetilde{R}_k :
    C_c^\infty(\mathbb{R}^n)
    \longrightarrow
    C^\infty(\mathbb{R}^n).
\end{equation}

For $\varepsilon>0$ and $h\in C_c^\infty(\mathbb{R}^n)$, writing
$R_k^\varepsilon$ in Fourier variables gives
\begin{equation}\label{eq:amp_res}
    \widetilde{R}_{k,\varepsilon} h(x)
    =
    \frac{1}{(2\pi)^n}
    \int_{\mathbb{R}^n}
    e^{ix\cdot\eta}
    \frac{\widehat{h}(\eta)}
    {|k\alpha+\eta|-k-i\varepsilon}
    \,d\eta.
\end{equation}

We also define the transport operator: for $h\in C_c^\infty(\mathbb{R}^n)$,
\begin{equation}\label{eq:transport_def}
  T_\alpha h(x)
  := i\int_{-\infty}^{x\cdot\alpha} h(\tau\alpha+x_\perp)\,d\tau,
  \qquad x_\perp = x-(x\cdot\alpha)\,\alpha.
\end{equation}
As shown in Lemma~\ref{lem:transport}, see Appendix~\ref{sec:appendix},
\begin{equation*}
    T_\alpha h(x)=\lim_{\varepsilon\to 0^+}
\mathcal{F}^{-1}\left[\frac{\widehat{h}(\cdot)}{\alpha\cdot(\cdot)-i\varepsilon}\right](x).
\end{equation*}

For $M\in\mathbb{N}$ and $\varphi$ in the Schwartz space
$\mathscr{S}(\mathbb{R}^n)$, we write
\begin{equation*}
    \|\varphi\|_{\mathscr{S}_M}
    :=
    \max_{|\beta|\leq M}\,
    \sup_{y\in\mathbb{R}^n}\,
    \langle y\rangle^{M}\,
    |\partial^\beta\varphi(y)|.
\end{equation*}

The following result, proved in Appendix~\ref{sec:appendix}, provides
a pointwise geometric-optics approximation for $\widetilde{R}_k$.

\begin{proposition}\label{lem:GO_ptwise}
For every compact set $K\subset\mathbb{R}^n$ and multi-index $\beta$ there exist $C_{K,\beta}>0$ and $M_{\beta}\in\mathbb{N}$ such that, for every $\alpha\in\mathbb{S}^{n-1}$, $h\in C_c^\infty(\mathbb{R}^n)$, and $k> 1$,
\begin{equation}\label{eq:GO_main}
  \sup_{x\in K}\left|
    \partial_x^\beta\!\left(\widetilde{R}_k\, h(x) - T_\alpha h(x)\right)
  \right|
  \leq \frac{C_{K,\beta}}{k}\|\widehat{h}\|_{\mathscr{S}_{M_{\beta}}}.
\end{equation}
\end{proposition}

Let $\alpha\in\mathbb{S}^{n-1}$ and $\chi\in C_c^\infty(\mathbb{R}^n)$. Define
\begin{equation}\label{eq:b0_def}
    b_0(x) := T_\alpha \chi(x)
    = i \int_{-\infty}^{x \cdot \alpha}
    \chi(\tau\alpha + x_\perp)\, d\tau,
    \qquad x_\perp = x - (x \cdot \alpha)\alpha.
\end{equation}
Note that $b_0 \in C^\infty(\mathbb{R}^n)$.

The uniform resolvent estimates of Section~\ref{sec:direct} transfer to the
conjugated operators. Let $\widetilde{P}^0_k$ denote the Fourier multiplier
with symbol $|k\alpha+\eta|-k$, so that
\begin{equation}\label{eq:tildR_tildP0}
    \widetilde{P}^0_k = e^{-ik\alpha\cdot (\cdot)}\,P^0_k\, e^{ik\alpha\cdot(\cdot)}
    \qquad
    \text{and}
    \qquad
    \widetilde R_{k,\varepsilon}
    =
    (\widetilde P_k^0-i\varepsilon)^{-1}
\end{equation}
for $\varepsilon>0$. We also set $\widetilde{P}_k := \widetilde{P}^0_k + V_k$.

\begin{lemma}\label{lem:conj_resolvent}
Let $n\geq2$, $\delta>\frac12$, and let $l\geq0$ be an integer.
\begin{enumerate}
\item[(i)] For all $k\geq1$, $\varepsilon\geq0$, $\alpha\in\mathbb{S}^{n-1}$,
and $f\in H^l_\delta(\mathbb{R}^n)$,
\begin{equation}\label{eq:conj_Hl_est}
\|\widetilde{R}_{k,\varepsilon}f\|_{H^l_{-\delta}}
\leq C_{\delta,l}\,\|f\|_{H^l_\delta},
\end{equation}
where $\widetilde{R}_{k,0}:=\widetilde{R}_k$ and $C_{\delta,l}$ is the
constant of Corollary~\ref{cor:free_frac_Hl}; in particular, it does not
depend on $k$, $\varepsilon$, or $\alpha$.
\item[(ii)] There exists $k_{\rho,\delta,l}>0$ such that for
$k>k_{\rho,\delta,l}$ and $f\in H^l_\delta(\mathbb{R}^n)$, the limit
\begin{equation*}
\widetilde{G}_k f := \lim_{\varepsilon\to0^+}(\widetilde{P}_k - i\varepsilon)^{-1}f
\end{equation*}
exists in $H^l_{-\delta}(\mathbb{R}^n)$ and satisfies
\begin{equation}\label{eq:conj_LAP_est}
\|\widetilde{G}_k f\|_{H^l_{-\delta}} \leq C_{\rho,l,\delta}\,\|f\|_{H^l_\delta},
\end{equation}
with $C_{\rho,l,\delta}$ and $k_{\rho,\delta,l}$ independent of $\alpha$. Moreover,
\begin{equation}\label{eq:conj_LAP_identity}
e^{ik\alpha\cdot(\cdot)}\,\widetilde{G}_k f
= (P_k - i0)^{-1}\bigl(e^{ik\alpha\cdot(\cdot)}f\bigr).
\end{equation}
\end{enumerate}
\end{lemma}

\begin{proof}
(i) Let $\varepsilon>0$. Since $|e^{\pm ik\alpha\cdot x}|=1$, multiplication by $e^{\pm ik\alpha\cdot x}$
is an isometry of $L^2_{\pm\delta}(\mathbb{R}^n)$. Therefore, by \eqref{eq:conj_identity}, we conclude that
\begin{equation*}
    \|\widetilde{R}_{k,\varepsilon}\|_{L^2_\delta\to L^2_{-\delta}} = \|R^\varepsilon_k\|_{L^2_\delta\to L^2_{-\delta}}.
\end{equation*}
Thus, Lemma~\ref{lem:free_frac}
holds for $\widetilde{R}_{k,\varepsilon}$. Since
$\widetilde{R}_{k,\varepsilon}$ is a Fourier multiplier, it commutes with
$(1-\Delta)^{l/2}$. Hence, by repeating steps of the proof of Corollary~\ref{cor:free_frac_Hl}, we obtain \eqref{eq:conj_Hl_est} for $\varepsilon>0$. 
Letting $\varepsilon\to0^+$ in \eqref{eq:conj_Hl_est} and using
\eqref{eq:amp_res_convergence}, we obtain \eqref{eq:conj_Hl_est}
for $\varepsilon=0$.

\emph{(ii)} Due to \eqref{eq:tildR_tildP0}, we have
\begin{equation}\label{eq:tPk_tPk0_tRk}
    \widetilde P_k-i\varepsilon
    =
    (\widetilde P_k^0-i\varepsilon)
    \bigl(I+\widetilde R_{k,\varepsilon}V_k\bigr).
\end{equation}
By part (i) and \eqref{eq:Vk_bound}, for sufficiently large $k>0$,
\begin{equation*}
    \|\widetilde R_{k,\varepsilon}V_k\|_{H^l_{-\delta}\to H^l_{-\delta}}
    \leq
    \frac{C_{\rho,l,\delta}}{k}<1,
\end{equation*}
uniformly in $\varepsilon\geq0$ and
$\alpha\in\mathbb S^{n-1}$.

Therefore, for $\varepsilon>0$ and sufficiently large $k>0$,
the operator $I+\widetilde R_{k,\varepsilon}V_k$ is invertible on
$H^l_{-\delta}(\mathbb R^n)$, with
\[
\bigl(I+\widetilde R_{k,\varepsilon}V_k\bigr)^{-1}
=
\sum_{j=0}^\infty
(-\widetilde R_{k,\varepsilon}V_k)^j.
\]
Consequently, by \eqref{eq:tPk_tPk0_tRk} and \eqref{eq:tildR_tildP0},
\begin{equation}\label{eq:conj_neumann}
(\widetilde P_k-i\varepsilon)^{-1}f
=
\sum_{j=0}^\infty
(-\widetilde R_{k,\varepsilon}V_k)^j
\widetilde R_{k,\varepsilon}f.
\end{equation}
The series converges in $H^l_{-\delta}(\mathbb R^n)$, and part (i) gives
\[
\|(\widetilde P_k-i\varepsilon)^{-1}f\|_{H^l_{-\delta}}
\leq
\frac{C_{\delta,l}}
{1-C_{\rho,l,\delta}/k}
\|f\|_{H^l_\delta}
\leq
C_{\rho,l,\delta}\|f\|_{H^l_\delta},
\]
for sufficiently large $k>0$.

By \eqref{eq:amp_res_convergence}, one can prove
inductively that, for every fixed $j\geq0$,
\begin{equation*}
    (-\widetilde R_{k,\varepsilon}V_k)^j
    \widetilde R_{k,\varepsilon}f
    \longrightarrow
    (-\widetilde R_kV_k)^j\widetilde R_kf
\end{equation*}
in $H^l_{-\delta}(\mathbb R^n)$ as $\varepsilon\rightarrow 0^+$.
Since the Neumann series in \eqref{eq:conj_neumann} has a geometric tail
bounded uniformly in $\varepsilon$, we may pass to the limit and obtain
\[
\widetilde G_k f
:=
\lim_{\varepsilon\to0^+}
(\widetilde P_k-i\varepsilon)^{-1}f
=
\sum_{j=0}^\infty
(-\widetilde R_kV_k)^j\widetilde R_kf
\]
in $H^l_{-\delta}(\mathbb R^n)$. The preceding uniform estimate also gives \eqref{eq:conj_LAP_est}.

Finally, by recalling the definitions of $\widetilde{P}_k$ and $\widetilde{P}^0_k$ , we write
\begin{equation*}
    \widetilde{P}_k = e^{-ik\alpha\cdot (\cdot)}\,P^0_k\, e^{ik\alpha\cdot(\cdot)} + e^{-ik\alpha\cdot (\cdot)}\,V_k\, e^{ik\alpha\cdot(\cdot)} = e^{-ik\alpha\cdot (\cdot)}\,P_k\, e^{ik\alpha\cdot(\cdot)}.
\end{equation*}
Thus, for every $\varepsilon>0$,
\[
e^{ik\alpha\cdot(\cdot)}
(\widetilde P_k-i\varepsilon)^{-1}
=
(P_k-i\varepsilon)^{-1}e^{ik\alpha\cdot(\cdot)}.
\]
Applying this identity to $f$ and letting $\varepsilon\to0^+$ gives \eqref{eq:conj_LAP_identity}.
\end{proof}

\begin{lemma}\label{lem:Born_GO}
    Let $\alpha\in\mathbb{S}^{n-1}$, $\chi\in C_c^\infty(\mathbb{R}^n)$, $f=e^{ik\alpha\cdot(\cdot)}\chi$, and $x\in \mathbb{R}^n$. Then, there exists $k_{\rho}>0$ such that for $k>k_{\rho}$, the outgoing solution
    $\psi^{f}$ satisfies
    \begin{equation}\label{eq:Born_GO}
        \psi^{f}(x)
        =  \psi_0(x) - e^{ik\alpha\cdot x}\,
        \frac{g^2}{k - \Omega}\, T_\alpha(\rho\, b_0)(x)
        + O_{\rho,\chi,x}(k^{-2}),
    \end{equation}
    where $\psi_0 = R_k^0 f$ and
    $b_0$ is defined by \eqref{eq:b0_def}.
\end{lemma}

\begin{proof}
Fix $x\in \mathbb{R}^n$. We set $\delta = 3/4$ and $l = [n/2]+1$. Let $k_\rho\geq\max\{1,2\Omega\}$ be sufficiently large so that the conclusions of
Lemmas~\ref{lem_LAP} and \ref{lem:conj_resolvent} hold for $k>k_\rho$.

Now, we assume that $k>k_{\rho}$. By Lemma~\ref{lem_LAP},
\begin{equation}\label{eq:Neumann_psi}
    \psi^{f}
    =
    \sum_{m=0}^\infty (-R_k^0 V_k)^m \psi_0
    \quad \text{in } H^l_{-\delta}(\mathbb{R}^n).
\end{equation}
Fix a closed ball $B$ with $x\in B$. Since $l>n/2$, the Sobolev embedding
$H^l_{-\delta}(\mathbb{R}^n)\hookrightarrow C^0(B)$ holds, so the series
\eqref{eq:Neumann_psi} converges in $C^0(B)$. As each term is continuous,
the identity below may be evaluated pointwise at $x$, and any tail is
controlled in $C^0(B)$ by its $H^l_{-\delta}$-norm.

Split the series
\begin{equation}\label{eq:split}
    \psi^{f}
    =
    \psi_0
    - R_k^0 V_k \psi_0
    + \sum_{m=2}^{\infty} (-1)^m (R_k^0 V_k)^m \psi_0.
\end{equation}
We recall \eqref{eq:conj_identity} and the definition of $\psi_0$ to conclude that
\begin{equation}\label{eq:sec_term0}
    R_k^0 V_k \psi_0 = e^{ik\alpha\cdot (\cdot)} \frac{g^2}{k - \Omega} \, \widetilde{R}_k \, \rho \,\widetilde{R}_k \, \chi
\end{equation}
By Proposition~\ref{lem:GO_ptwise} applied
with $h = \rho\widetilde{R}_k \chi$ and $\beta=0$, we obtain
\begin{equation*}
\widetilde{R}_k \, \left(\rho \,\widetilde{R}_k \, \chi\right) (x)
= T_\alpha\left(\rho \,\widetilde{R}_k\chi\right)(x) + O\left( k^{-1} \|\widehat{\rho\widetilde{R}_k \chi}\|_{\mathscr{S}_M}\right).
\end{equation*}
for some $M\in \mathbb{N}$. Since $\rho\widetilde{R}_k\chi$ is smooth and supported in $\Sigma$,
\[
\|\widehat{\rho\widetilde{R}_k\chi}\|_{\mathscr S_M}
\leq
C_{\rho,M}
\|\rho\widetilde{R}_k\chi\|_{C^M(\Sigma)}
\leq
C_{\rho,M}
\|\widetilde{R}_k\chi\|_{C^M(\Sigma)}.
\]
By Sobolev embedding and Lemma~\ref{lem:conj_resolvent}(i), applied with
a sufficiently large Sobolev order, we obtain $\|\widetilde{R}_k\chi\|_{C^M(\Sigma)}
\leq C_{\chi,M}$, uniformly in $k$ and $\alpha$. Therefore,
\begin{equation*}
\widetilde{R}_k \, \left(\rho \,\widetilde{R}_k \, \chi\right) (x)
= T_\alpha\left(\rho \,\widetilde{R}_k\chi\right)(x) + O_{\rho,\chi,x}(k^{-1}).
\end{equation*}
Moreover, since $\rho$ is supported in $\Sigma$,
\begin{multline*}
   \left|T_\alpha\left(\rho \,\widetilde{R}_k\chi - \rho b_0\right)(x) \right| = \left|T_\alpha\left(\rho \,\widetilde{R}_k\chi - \rho \, T_\alpha\chi\right)(x) \right|\\
   \leq \int_{\mathbb{R}}\bigl|\rho(x+t\alpha)\bigr|\,dt\;
\sup_{\Sigma}\left|\widetilde{R}_k\chi - T_\alpha\chi\right| = O_{\rho,\chi}(k^{-1}).
\end{multline*}
Combining the last two estimates and \eqref{eq:sec_term0} gives
\begin{equation}\label{eq:second_term}
    R_k^0 V_k \psi_0(x) = e^{ik\alpha\cdot x}\,
        \frac{g^2}{k - \Omega}\, T_\alpha(\rho\, b_0)(x)
        + O_{\rho,\chi,x}(k^{-2}).
\end{equation}

Next, we estimate the remainder. We write
\begin{equation*}
    r := \sum_{m=2}^{\infty} (-1)^m (R_k^0 V_k)^m \psi_0 = \sum_{m=0}^{\infty} (-1)^m (R_k^0 V_k)^m R_k^0 \left(V_k R_k^0 V_k R_k^0 f\right).
\end{equation*}
We denote $\psi_1 = R_k^0 V_k R_k^0 f$. Applying \eqref{eq:LAP_def} and
\eqref{eq:psi_equals_Nseries} with $f$ replaced by $V_k\psi_1$ gives
\begin{equation}\label{eq:r_formula}
r = (P_k - i0)^{-1}(V_k\psi_1).
\end{equation}
Therefore, by \eqref{eq:conj_LAP_identity},
\begin{equation*}
    r = e^{ik\alpha\cdot x}\,\widetilde{G}_k(V_k\, e^{-ik\alpha\cdot (\cdot)} \, \psi_1).
\end{equation*}
By \eqref{eq:conj_LAP_est} and \eqref{eq:Vk_bound}, we estimate
\begin{equation*}
\|\widetilde{G}_k(V_k\, e^{-ik\alpha\cdot (\cdot)} \, \psi_1)\|_{H^l_{-\delta}}
\leq C_\rho \,\|V_k\, e^{-ik\alpha\cdot (\cdot)} \, \psi_1\|_{H^l_\delta}
\leq \frac{C_\rho}{k}\,\|e^{-ik\alpha\cdot (\cdot)} \, \psi_1\|_{H^l_{-\delta}},
\end{equation*}
By \eqref{eq:conj_Hl_est} and \eqref{eq:Vk_bound},
\begin{multline*}
    \|e^{-ik\alpha\cdot (\cdot)} \, \psi_1\|_{H^l_{-\delta}}  
    = 
    \|e^{-ik\alpha\cdot (\cdot)} \, R_k^0 \, V_k \, R_k^0 \, e^{ik\alpha\cdot (\cdot)} \, \chi\|_{H^l_{-\delta}} 
    = 
    \|\widetilde{R}_k \, V_k \, \widetilde{R}_k \, \chi\|_{H^l_{-\delta}} \leq \frac{C_\rho}{k} \, \|\chi\|_{H^l_{\delta}} 
\end{multline*}
Since $l>n/2$, the Sobolev embedding gives
\begin{equation*}
|r(x)| = |\widetilde{G}_k(V_k\, e^{-ik\alpha\cdot (\cdot)} \, \psi_1)(x)|
\leq C_x\,\|\widetilde{G}_k(V_k\, e^{-ik\alpha\cdot (\cdot)} \, \psi_1)\|_{H^l_{-\delta}}
= O_{\rho,\chi,x}(k^{-2}).
\end{equation*}
Substituting this and \eqref{eq:second_term} into \eqref{eq:split} completes the proof.
\end{proof}

\subsection{Recovery of the X-ray transform}
Let us fix $z\in S$ and $\alpha \in \mathbb{S}^{n-1}$.
For $x\in \mathbb{R}^n$, we write
\begin{equation*}
    x = s\alpha + x_\perp,
\end{equation*}
where
\begin{equation*}
    s = x\cdot\alpha \in \mathbb{R},
    \qquad
    x_\perp = x-(x\cdot\alpha)\alpha.
\end{equation*}
Since $z\in S$ and $S$ is open, there exist
$\varepsilon_0>0$ and $\delta_0>0$ such that
\begin{equation}\label{def_tube}
    \mathcal{T}_{\mathrm{src}}
    = \{ z + s\alpha + x_\perp : |s| < \varepsilon_0,\ |x_\perp| < 2\delta_0 \}
    \subset S.
\end{equation}
Let $\chi_\parallel \in C_c^\infty(\mathbb{R})$ and $\chi_\perp \in C_c^\infty(\alpha^\perp)$ be non-negative functions such that $\chi_\parallel(0) = \chi_\perp(0) = 1$,  and
\begin{equation*}
    \operatorname{supp}(\chi_\parallel) \subset (-\varepsilon_0,\varepsilon_0),
    \qquad
    \operatorname{supp}(\chi_\perp) \subset \{ |y| < 2\delta_0 \}.
\end{equation*}
Define
\begin{equation}\label{def_chi}
    \chi(x)
    = \chi_\parallel(x\cdot\alpha)\,\chi_\perp(x_\perp),
    \qquad
    \chi_{z,\alpha}(x) = \chi(x - z),
\end{equation}
and set
\begin{equation}\label{def_source}
    f_{k,z}^\alpha(x) = e^{ik\alpha\cdot x}\, \chi_{z,\alpha}(x).
\end{equation}
Note that $f_{k,z}^\alpha \in C_c^\infty(S)$.
\begin{theorem}\label{thm:determination}
    Let $n \geq 2$ and $k_\rho$ be the threshold from Lemma \ref{lem_LAP} with $l=[n/2]+1$ and $\delta=3/4$. Assume that
    \begin{equation*}
        \Sigma \cap (S\cup W) = \varnothing.
    \end{equation*}
    Assume that $z_1, z_2\in S$ and $x^* \in W$ satisfy
    \begin{equation*}
        [z_2,x^*] \cap \Sigma = \varnothing.
    \end{equation*}
    Then the data $\{\Lambda_k^{\mathrm{src}}\}_{k>k_\rho}$ determines the integral
    \begin{equation*}
        \int_0^1 \rho\bigl(z_1 + s(x^* - z_1)\bigr)\,ds.
    \end{equation*}
\end{theorem}

\begin{proof}
If $x^*=z_1$, the integral in question equals $\rho(z_1)=0$, since $z_1\in S$ and $S\cap\Sigma=\varnothing$, and there is nothing to prove. If $x^*=z_2$, then, since $S$ is open, $x^*\notin\Sigma$, and $\Sigma$ is closed, we may
replace $z_2$ by a nearby point $z_2'\in S\setminus\{x^*\}$ satisfying $[z_2',x^*]\cap\Sigma=\varnothing$. Hence we may assume $x^*\neq z_1,z_2$.

Let $\alpha_j = (x^*-z_j)/|x^*-z_j|$ and $t_j^* = |x^*-z_j|$ for $j=1,2$,
so that $z_j + t_j^*\alpha_j = x^*$. Let $\chi_1$, $\chi_2\in C_c^\infty(S)$ be cutoff functions as in \eqref{def_chi}, corresponding to the directions $\alpha_1,\alpha_2$ and points $z_1,z_2$, with parameters $\varepsilon_{0,1},\varepsilon_{0,2}>0$ and corresponding tubes $\mathcal T_{\mathrm{src},1},\mathcal T_{\mathrm{src},2}\subset S$ as in \eqref{def_tube}, respectively. We choose $\varepsilon_{0,j}<t_j^*$, $j=1,2$. Define
\begin{equation*}
    f_j = f_{k,z_j}^{\alpha_j} = e^{ik\alpha_j\cdot x}\chi_j(x),
    \qquad j=1,2,
\end{equation*}
and let $\psi^{f_j}$ denote the corresponding outgoing solutions. In the argument below we take $k>k_*$, where $k_*\geq k_\rho$ is chosen
large enough so that Lemma~\ref{lem:Born_GO} applies to both sources
$f_1$ and $f_2$.

The polarization identity gives
\begin{equation}\label{eq:polarization}
    \psi^{f_1}(w)\,\overline{\psi^{f_2}(w)}
    = \frac{1}{4}\sum_{j=0}^{3}i^{j}\,
    \Lambda_k^{\mathrm{src}}(f_1+i^{j}f_2)(w),
    \qquad w\in W.
\end{equation}
Therefore the data determines
$\psi^{f_1}(w)\,\overline{\psi^{f_2}(w)}$ for $w\in W$.

We now evaluate at $w=x^*$.
By Lemma~\ref{lem:Born_GO}, since $\chi_j$ is supported in $S$
and $x^*$ lies beyond the support of each $\chi_j$ in the $\alpha_j$-direction:
\begin{equation*}
    b_{0,j}(x^*) = T_{\alpha_j}\chi_j(x^*) = i\int_{-\infty}^{\infty}\chi_j(\tau\alpha_j+(x^*)_{\perp,j})\,d\tau
    = iC_{\chi_j},
    \qquad C_{\chi_j} = \|\chi_{\parallel,j}\|_{L^1},
\end{equation*}
where $(x^*)_{\perp,j} = x^* - (x^*\cdot\alpha_j)\alpha_j$. Since $x^*-z_1=t_1^*\alpha_1$, along the ray
$z_1+\tau\alpha_1$ we have
\[
b_{0,1}(z_1+\tau\alpha_1)
=
i\int_{-\infty}^{\tau}\chi_{\parallel,1}(s)\,ds.
\]
Hence
\[
b_{0,1}(z_1+\tau\alpha_1)=0,
\qquad \tau\leq -\varepsilon_{0,1},
\]
whereas, for $-\varepsilon_{0,1}<\tau\leq 0$, the point
$z_1+\tau\alpha_1$ belongs to
$\mathcal T_{\mathrm{src},1}\subset S$. Since
$S\cap\Sigma=\emptyset$, we have
\[
\rho(z_1+\tau\alpha_1)=0,
\qquad -\varepsilon_{0,1}<\tau\leq 0.
\]
Therefore,
\[
\rho(z_1+\tau\alpha_1)
b_{0,1}(z_1+\tau\alpha_1)=0,
\qquad \tau\leq 0.
\]

Moreover, since $\varepsilon_{0,1}<t_1^*$,
\[
b_{0,1}(z_1+\tau\alpha_1)
=
iC_{\chi_1},
\qquad
\varepsilon_{0,1}\leq \tau\leq t_1^*,
\]
and for $0\leq\tau<\varepsilon_{0,1}$ we have
$\rho(z_1+\tau\alpha_1)=0$. Consequently,
\[
\rho(z_1+\tau\alpha_1)
b_{0,1}(z_1+\tau\alpha_1)
=
iC_{\chi_1}\rho(z_1+\tau\alpha_1),
\qquad 0\leq\tau\leq t_1^*.
\]
Hence, by the definition of $T_{\alpha_1}$
and $V_k=\frac{g^2}{k-\Omega}\rho$,
\begin{multline*}
    T_{\alpha_1}(V_k b_{0,1})(x^*)
    = i\int_{-\infty}^{t_1^*}
       V_k(z_1+\tau\alpha_1)\,b_{0,1}(z_1+\tau\alpha_1)\,d\tau\\
       = i\,(iC_{\chi_1})\,\frac{g^2}{k-\Omega}
       \int_0^{t_1^*}\rho(z_1+\tau\alpha_1)\,d\tau
    = -C_{\chi_1}\frac{g^2}{k-\Omega}
       \int_0^{t_1^*}\rho(z_1+\tau\alpha_1)\,d\tau.
\end{multline*}

Writing $x^*=z_2+t_2^*\alpha_2$, we have
\begin{equation*}
    T_{\alpha_2}(V_k b_{0,2})(x^*)
    =
    i\int_{-\infty}^{t_2^*}
    V_k(z_2+s\alpha_2)b_{0,2}(z_2+s\alpha_2)\,ds.
\end{equation*}
The integrand vanishes on the whole interval $(-\infty,t_2^*]$: for
$s\leq-\varepsilon_0$ this follows from the support of
$\chi_{\parallel,2}$; for $|s|<\varepsilon_0$ the point
$z_2+s\alpha_2$ lies in $\mathcal{T}_{\mathrm{src},2}\subset S$, where $\rho=0$;
and for $s\in[\varepsilon_0,t_2^*]$ the point lies on $[z_2,x^*]$,
which is disjoint from $\Sigma$ by hypothesis. Hence,
\begin{equation*}
    T_{\alpha_2}(V_k b_{0,2})(x^*)=0.
\end{equation*}
Applying \eqref{eq:Born_GO}:
\begin{align*}
    \psi^{f_1}(x^*)
    &= e^{ik\alpha_1\cdot x^*}\left(
    \widetilde{R}_k \chi_1(x^*)
    + C_{\chi_1}\frac{g^2}{k-\Omega}\int_0^{t_1^*}\rho(z_1+\tau\alpha_1)\,d\tau
    + O(k^{-2})\right),\\
    \psi^{f_2}(x^*)
    &= e^{ik\alpha_2\cdot x^*}\bigl(
    \widetilde{R}_k \chi_2(x^*) + O(k^{-2})\bigr).
\end{align*}
Therefore,
\begin{multline}\label{eq:cross_term}
    \psi^{f_1}(x^*) \overline{\psi^{f_2}(x^*)}
    = e^{ik(\alpha_1-\alpha_2)\cdot x^*}
    \Biggl(
    \widetilde{R}_k \chi_1(x^*) \overline{\widetilde{R}_k \chi_2(x^*)}\\
    + C_{\chi_1} \overline{\widetilde{R}_k \chi_2(x^*)}\frac{g^2}{k-\Omega}
    \int_0^{t_1^*}\rho(z_1+\tau\alpha_1)\,d\tau
    + O(k^{-2})
    \Biggr).
\end{multline}
The left-hand side is known from the data via \eqref{eq:polarization}.
The factor $e^{ik(\alpha_1-\alpha_2)\cdot x^*}$ has modulus one and depends only
on $k$, $\alpha_1$, $\alpha_2$, and $x^*$, hence is known. The first term
$\widetilde{R}_k\chi_1(x^*)\overline{\widetilde{R}_k\chi_2(x^*)}$ is independent of
$\rho$ and known. Therefore, by \eqref{eq:cross_term}, the data determine
\begin{multline*}
    \frac{k-\Omega}{g^2}
    \left(
    e^{-ik(\alpha_1-\alpha_2)\cdot x^*}\,
    \psi^{f_1}(x^*)\overline{\psi^{f_2}(x^*)}
    - \widetilde{R}_k\chi_1(x^*)\overline{\widetilde{R}_k\chi_2(x^*)}
    \right)\\
    = C_{\chi_1}\overline{\widetilde{R}_k\chi_2(x^*)}
    \int_0^{t_1^*}\rho(z_1+\tau\alpha_1)\,d\tau + O(k^{-1}).
\end{multline*}
By Proposition~\ref{lem:GO_ptwise},
$\widetilde{R}_k\chi_2(x^*)\to T_{\alpha_2}\chi_2(x^*)=b_{0,2}(x^*)=iC_{\chi_2}$ as
$k\to\infty$, so
\begin{equation*}
    C_{\chi_1}\overline{\widetilde{R}_k\chi_2(x^*)}
    \to -iC_{\chi_1}C_{\chi_2}\neq 0.
\end{equation*}
Letting $k\to\infty$ therefore gives
\begin{multline*}
    \int_0^{t_1^*}\rho(z_1+\tau\alpha_1)\,d\tau
    \\
    = \frac{1}{-iC_{\chi_1}C_{\chi_2}}
      \lim_{k\to\infty}
      \frac{k-\Omega}{g^2}
    \left(
    e^{-ik(\alpha_1-\alpha_2)\cdot x^*}\,
    \psi^{f_1}(x^*)\overline{\psi^{f_2}(x^*)}
    - \widetilde{R}_k\chi_1(x^*)\overline{\widetilde{R}_k\chi_2(x^*)}
    \right).
\end{multline*}
Since the right-hand side is computed entirely from the data, the integral is
determined.

Finally, since
\begin{equation*}
    x^*-z_1=t_1^*\alpha_1,
\end{equation*}
the change of variables $\tau=s t_1^*$ gives
\begin{equation*}
    \int_0^{t_1^*}\rho(z_1+\tau\alpha_1)d\tau
    =
    t_1^*
    \int_0^1
    \rho(z_1+s(x^*-z_1))ds.
\end{equation*}
The number $t_1^*=|x^*-z_1|$ is known. Hence the data determine
\begin{equation*}
    \int_0^1
    \rho(z_1+s(x^*-z_1))ds,
\end{equation*}
as claimed.
\end{proof}

\subsection{Determination from the source-to-solution map}
We now prove Theorem~\ref{thm:source_recovery}. 
Theorem~\ref{thm:determination} shows that the source-to-solution data $\{\Lambda_k^{\mathrm{src}}\}_{k>k_\rho}$ determine the integral of $\rho$ along a segment $[z_1,x^*]$ which
intersects $\Sigma$, provided that there is a reference source $z_2\in S$
such that the segment $[z_2,x^*]$ avoids $\Sigma$. 
Condition~\ref{cond:geometric} ensures that such line integrals are available for a sufficiently rich family of rays. We use this family of line integrals to determine $\rho$ by injectivity of the X-ray transform.

We will use the following elementary auxiliary lemma.
\begin{lemma}\label{stab}
Let $A\subset\mathbb{R}^n$ be compact, let $B\subset\mathbb{R}^n$ be open, and let $x\in\mathbb{R}^n$. Assume that $A\subset (x;B]$. Then there exists $\varepsilon>0$ such that $A\subset (q;B]$ for every $q\in B_\varepsilon(x)$.
\end{lemma}
\begin{proof}
For $q,a\in\mathbb{R}^n$, $a\in(q;B]$ if and only if $(1-s)q+sa\in B$ for some $s\geq 1$. Indeed, this follows by setting $s=1/t$ in
$a=(1-t)q+tb$.

For each $a\in A$, choose $s_a\geq 1$ such that 
\[
(1-s_a)x+s_a a\in B,
\]
and define
\[
F_a:\mathbb{R}^n\times\mathbb{R}^n\to\mathbb{R}^n,
\qquad
F_a(q,a')=(1-s_a)q+s_a a'.
\]
The map $F_a$ is continuous, and $F_a(x,a)\in B$. Since $B$ is open,
there exist neighborhoods $V_a$ of $x$ and $U_a$ of $a$ such that $F_a(V_a\times U_a)\subset B$. Thus, for every $q\in V_a$ and $a'\in U_a$, we have $a'\in(q;B]$.

Since $A$ is compact, finitely many of the sets $U_a$ cover $A$.
Let $U_{a_1},\ldots,U_{a_N}$ be such a finite subcover and set $V=\bigcap_{j=1}^N V_{a_j}$. Then $A\subset(q;B]$ for every $q\in V$. Choosing $\varepsilon>0$
such that $B_\varepsilon(x)\subset V$ proves the claim.
\end{proof}

Now we are ready to prove Theorem~\ref{thm:source_recovery}.

\begin{proof}[Proof of Theorem~\ref{thm:source_recovery}]
Since $\Sigma$ is compact and $\Sigma\subset(x;W']$, Lemma~\ref{stab} implies that there exists $\varepsilon>0$ such that
\begin{equation}\label{eq_l:cond}
\Sigma\subset(q;W']
\qquad\text{for all }q\in B_\varepsilon(x).
\end{equation}

Let us fix $\varepsilon>0$ such that $B_\varepsilon(x) \subset S$ and \eqref{eq_l:cond} holds. Fix $z\in B_\varepsilon(x)$. By Condition~\ref{cond:geometric}, $[y,p]\cap\Sigma=\varnothing$ for any $p\in W'$. Then Theorem~\ref{thm:determination}, applied with $z_1=z$, $z_2=y$, and $x^*=p\in W'$, shows that the data determine
\begin{equation*}
    \int_0^1 \rho(z+s(p-z)) ds.
\end{equation*}

Finally, since this holds for every $z \in B_\varepsilon(x)$, and $B_\varepsilon(x)$ is an open, non-empty set, and since $\Sigma \subset (z; W']$ as we established above, we conclude from Theorem 6.1 from \cite{JatharKarKrishnanPattar} that the given data determine the density $\rho$.
\end{proof}

\section{The far-field inverse problems}\label{sec:farfield}

\subsection{Determination from renormalized far-field intensity data}
We now prove that the renormalized far-field intensity data determine the
density. The proof proceeds in two steps. First, we show that the oscillatory
asymptotics of the renormalized intensity determine the scattering amplitude
away from the forward direction. Second, we use the high-frequency behaviour of
the scattering amplitude to recover the Fourier transform of $\rho$.

\begin{proposition}\label{prop:justification_Ak}
Let $k_\rho$ be the threshold from Lemma \ref{lem_LAP} with $l=[n/2]+1$ and $\delta=3/4$. Let $k>k_\rho$ and let $\alpha,\theta\in\mathbb S^{n-1}$ satisfy
$\theta\neq\alpha$. Then $\mathcal I_\rho(\alpha,k,\theta,\cdot)$ determines the scattering amplitude $A_k(\theta,\alpha)$.
More precisely,
\begin{equation}\label{eq:Ak_recovery_from_Irho}
    A_k(\theta,\alpha)
    =
    \lim_{R\to\infty}
    \frac{1}{R}
    \int_R^{2R}
    \mathcal I_\rho(\alpha,k,\theta,r)
    e^{-ikr(1-\alpha\cdot\theta)}\,dr .
\end{equation}
\end{proposition}

\begin{proof}
Fix $k>k_\rho$ and $\alpha,\theta\in\mathbb S^{n-1}$ with
$\theta\neq\alpha$. Set
\begin{equation}
    \phi := k(1-\alpha\cdot\theta).
\end{equation}
Since $\theta\neq\alpha$, we have $\alpha\cdot\theta<1$, and therefore
$\phi\neq 0$.

By \eqref{eq:Irho_asymptotics}, there exists a function
$\varepsilon(r)=o(1)$ as $r\to\infty$ such that
\begin{equation}\label{eq:Irho_epsilon_expansion}
    \mathcal I_\rho(\alpha,k,\theta,r)
    =
    2\Re\bigl(e^{i\phi r}A_k(\theta,\alpha)\bigr)
    +
    \varepsilon(r),
    \qquad r\to\infty.
\end{equation}
Moreover,
\begin{equation}
    \Re\bigl(e^{i\phi r}A_k(\theta,\alpha)\bigr)
    =
    \frac{1}{2}  e^{i\phi r}A_k(\theta,\alpha)
    +
    \frac{1}{2} e^{-i\phi r}\overline{A_k(\theta,\alpha)} .
\end{equation}
Multiplying by $2e^{-i\phi r}$, we get
\begin{equation}
    2\Re\bigl(e^{i\phi r}A_k(\theta,\alpha)\bigr)e^{-i\phi r}
    =
    A_k(\theta,\alpha)
    +
    \overline{A_k(\theta,\alpha)} e^{-2i\phi r}.
\end{equation}

Using \eqref{eq:Irho_epsilon_expansion}, we obtain
\begin{multline*}
    \frac{1}{R}\int_R^{2R}
    \mathcal I_\rho(\alpha,k,\theta,r)e^{-i\phi r}\,dr
    = A_k(\theta,\alpha)\\
    + 
    \overline{A_k(\theta,\alpha)}
    \frac{1}{R}\int_R^{2R}e^{-2i\phi r}\,dr
    +
    \frac{1}{R}\int_R^{2R} \varepsilon(r)e^{-i\phi r}\,dr .
\end{multline*}
We now verify that the last two terms tend to zero.

First, since $\phi\neq 0$,
\begin{equation}
    \frac{1}{R}\int_R^{2R} e^{-2i\phi r}\,dr
    =
    \frac{e^{-4i\phi R}-e^{-2i\phi R}}{-2i\phi R}
    \to 0,
    \qquad
    \text{as } R\to\infty .
\end{equation}
Second, since $\varepsilon(r)\to0$ as $r\to\infty$, we have
\begin{equation}
    \sup_{r\in[R,2R]}|\varepsilon(r)|\to0
    \qquad \text{as } R\to\infty.
\end{equation}
Consequently,
\begin{equation}
\left|
    \frac{1}{R}\int_R^{2R} \varepsilon(r)e^{-i\phi r}\,dr
    \right|
    \le
    \frac{1}{R}\int_R^{2R} |\varepsilon(r)|\,dr
    \le
    \sup_{r\in[R,2R]}|\varepsilon(r)|
    \to0 .
\end{equation}
Hence the averaged contribution of the $\varepsilon(r)$-term tends to zero as $R\to\infty$. Therefore
\begin{equation}
    \lim_{R\to\infty}
    \frac{1}{R}\int_R^{2R}
    \mathcal I_\rho(\alpha,k,\theta,r)e^{-i\phi r}\,dr
    =
    A_k(\theta,\alpha).
\end{equation}
Since $\phi=k(1-\alpha\cdot\theta)$, this proves
\eqref{eq:Ak_recovery_from_Irho}.
\end{proof}

The previous formula depends only on the asymptotic class of
$\mathcal I_\rho$ modulo terms tending to zero at infinity. We record this
consequence explicitly, since this is the form in which the result will be used for the inverse problem.

\begin{corollary}\label{cor:equivalence_class_determines_Ak}
Let $k_\rho$ be the threshold from Lemma \ref{lem_LAP} with $l=[n/2]+1$ and $\delta=3/4$. Let $k>k_\rho$ and let $\alpha,\theta\in\mathbb{S}^{n-1}$ satisfy $\theta\neq\alpha$.
If $\mathcal{I}_{\rho_1}$ and $\mathcal{I}_{\rho_2}$ are equivalent in the sense of \eqref{eq:equivalence_relation}, then
\begin{equation}\label{eq:Ak_equal_for_equivalent}
    A_k^{\rho_1}(\theta,\alpha) = A_k^{\rho_2}(\theta,\alpha),
\end{equation}
where $A_k^{\rho_j}(\theta,\alpha)$ denotes the scattering amplitude associated with $\rho_j$, $j=1,2$.

In particular, the equivalence class $[\mathcal I_\rho]_\infty$ determines
$A_k(\theta,\alpha)$ for every $k>k_\rho$ and every non-forward pair
$\theta\ne\alpha$.
\end{corollary}

\begin{proof}
Set $\phi := k(1-\alpha\cdot\theta)$. Since $\theta\neq\alpha$, we have $\phi\neq 0$.
Define
\begin{equation*}
    \delta(r) := \mathcal{I}_{\rho_1}(\alpha,k,\theta,r) - \mathcal{I}_{\rho_2}(\alpha,k,\theta,r).
\end{equation*}
By the equivalence relation \eqref{eq:equivalence_relation}, $\delta(r) \to 0$ as $r\to\infty$.

Applying Proposition~\ref{prop:justification_Ak} to $\rho_1$ and to $\rho_2$, we obtain
\begin{equation*}
    A_k^{\rho_j}(\theta,\alpha)
    =
    \lim_{R\to\infty}
    \frac{1}{R}\int_R^{2R}
    \mathcal{I}_{\rho_j}(\alpha,k,\theta,r)\,e^{-i\phi r}\,dr,
    \qquad j=1,2.
\end{equation*}
Subtracting these two limits,
\begin{equation*}
    A_k^{\rho_1}(\theta,\alpha) - A_k^{\rho_2}(\theta,\alpha)
    =
    \lim_{R\to\infty}
    \frac{1}{R}\int_R^{2R}
    \delta(r)\,e^{-i\phi r}\,dr.
\end{equation*}
We estimate
\begin{equation*}
    \left|
        \frac{1}{R}\int_R^{2R}\delta(r)\,e^{-i\phi r}\,dr
    \right|
    \le
    \frac{1}{R}\int_R^{2R}|\delta(r)|\,dr
    \le
    \sup_{r\in[R,2R]}|\delta(r)|.
\end{equation*}
Since $\delta(r)\to 0$ as $r\to\infty$, we have $\sup_{r\in[R,2R]}|\delta(r)|\to 0$ as $R\to\infty$. Hence
\begin{equation*}
    A_k^{\rho_1}(\theta,\alpha) - A_k^{\rho_2}(\theta,\alpha) = 0,
\end{equation*}
which proves \eqref{eq:Ak_equal_for_equivalent}.
\end{proof}

It remains to show that knowledge of these scattering amplitudes is enough to determine the density. 

\begin{lemma}\label{lem:amplitude_determines_rho}
There exists $k_\rho>0$ such that the scattering amplitudes
\begin{equation*}
    \{A_k(\theta,\alpha) : \theta,\alpha\in\mathbb{S}^{n-1},\, \theta\neq \alpha, \, k > k_\rho\}
\end{equation*}
uniquely determine the density $\rho$.
\end{lemma}

\begin{proof}
By \eqref{eq:amplitude} and $V_k = \frac{g^2}{k-\Omega}\rho$,
\begin{equation}\label{eq:Ak_expanded}
A_k(\theta,\alpha)
=
\frac{-2d_n g^2\, k^{(n-1)/2}}{k-\Omega}
\int_{\mathbb{R}^n}
e^{-ik\theta\cdot y}\,\rho(y)\,\psi_{k,\alpha}(y)\,dy.
\end{equation}
We substitute $\psi_{k,\alpha} = e^{ik\alpha\cdot y} + \psi^{\mathrm{sc}}_{k,\alpha}$
to obtain
\begin{equation}\label{eq:Ak_split}
A_k(\theta,\alpha)
=
\frac{-2d_n g^2\, k^{(n-1)/2}}{k-\Omega}
\bigl(\hat{\rho}(k(\theta - \alpha)) + E_k(\theta,\alpha)\bigr),
\end{equation}
where
\begin{equation*}
E_k(\theta,\alpha)
= \int_{\mathbb{R}^n} e^{-ik\theta\cdot y}\,\rho(y)\,
\psi^{\mathrm{sc}}_{k,\alpha}(y)\,dy.
\end{equation*}

We estimate the error term. Set $l=[n/2]+1$ and $\delta=3/4$, and let $k_\rho\geq\max\{1,2\Omega\}$ be sufficiently large the conclusions of Lemmas~\ref{lem_LAP} and \ref{lem:conj_resolvent} hold for $k>k_\rho$. By \eqref{eq:sc_equation} and \eqref{eq:conj_LAP_identity}, 
\begin{equation*}
e^{-ik\alpha\cdot (\cdot)}\,\psi^{\mathrm{sc}}_{k,\alpha} = -\widetilde{G}_k V_k.
\end{equation*} 
Hence, by Sobolev embedding ($l>n/2$) and
\eqref{eq:conj_LAP_est}, we estimate
\begin{multline*}
    \|\psi^{\mathrm{sc}}_{k,\alpha}\|_{L^\infty(\Sigma)}
    =
    \|e^{-ik\alpha\cdot(\cdot)}
    \psi^{\mathrm{sc}}_{k,\alpha}\|_{L^\infty(\Sigma)}
    \leq
    C\|e^{-ik\alpha\cdot(\cdot)}
    \psi^{\mathrm{sc}}_{k,\alpha}\|_{H^l_{-\delta}}\\
    =
    C\|\widetilde{G}_k V_k\|_{H^l_{-\delta}}
    \leq
    C_\rho \|V_k\|_{H^l_\delta}
    \leq
    \frac{C_\rho}{k}.
\end{multline*}

Since $\rho$ is supported in $\Sigma$, it follows that
\begin{equation}\label{eq:Ak_split_bound_reminder}
|E_k(\theta,\alpha)|
\leq
\|\rho\|_{L^1(\Sigma)}
\|\psi^{\mathrm{sc}}_{k,\alpha}\|_{L^\infty(\Sigma)}
\leq
\frac{C_{\rho}}{k}.
\end{equation}
We emphasize that the constant in the bound $|E_k(\theta,\alpha)|\le C_\rho/k$ is independent of $\theta$ and $\alpha$:
the factor $e^{-ik\theta\cdot y}$ has modulus one, and the constants in
Lemma~\ref{lem:conj_resolvent} do not depend on $\alpha\in\mathbb{S}^{n-1}$. In
particular, the bound remains valid at the $k$-dependent pair
$(\theta_k,\alpha_k)$ constructed below.

Let $\xi\in\mathbb R^n\setminus\{0\}$ be arbitrary. Since $n\geq 2$, we may choose a unit vector
\begin{equation*}
    \omega \in \mathbb{S}^{n-1},
    \qquad
    \omega \cdot \xi = 0.
\end{equation*}
For every sufficiently large $k$ with $k > |\xi|/2$, define
\begin{equation*}
    \theta_k
    :=
    \frac{\xi}{2k}
    +
    \sqrt{1-\frac{|\xi|^2}{4k^2}}\,\omega,
    \qquad
    \alpha_k
    :=
    -\frac{\xi}{2k}
    +
    \sqrt{1-\frac{|\xi|^2}{4k^2}}\,\omega.
\end{equation*}
Then $\theta_k,\alpha_k \in \mathbb{S}^{n-1}$ and
\begin{equation*}
    k(\theta_k-\alpha_k)=\xi.
\end{equation*}
Since $\xi\neq0$, we have $\theta_k-\alpha_k=\xi/k\neq0$, so the pair
$(\theta_k,\alpha_k)$ is non-forward and $A_k(\theta_k,\alpha_k)$ is determined
by the data for $k>\max\{k_\rho,|\xi|/2\}$. Substituting $(\theta_k,\alpha_k)$
into \eqref{eq:Ak_split} and using $E_k(\theta_k,\alpha_k)=O(k^{-1})$, which
holds by the $\theta,\alpha$-uniformity just established, we obtain
\begin{equation}\label{eq:Ak_high_freq_here}
    \frac{k-\Omega}{-2d_n g^2 k^{(n-1)/2}}\,A_k(\theta_k,\alpha_k)
    =
    \hat{\rho}(\xi)
    +
    O(k^{-1}).
\end{equation}
Therefore,
\begin{equation*}
    \hat{\rho}(\xi)
    =
    \lim_{k\to\infty}
    \frac{k-\Omega}{-2d_n g^2 k^{(n-1)/2}}\,A_k(\theta_k,\alpha_k).
\end{equation*}
Since $\xi \in \mathbb{R}^n\setminus\{0\}$ was arbitrary, the scattering amplitudes determine $\widehat{\rho}(\xi)$ for every
$\xi\ne0$. Since $\rho\in C_c^\infty(\mathbb R^n)$, the function
$\widehat{\rho}$ is continuous, and therefore $\widehat{\rho}(0)$ is also
determined by taking the limit $\xi\to 0$. Hence, the data determine $\rho$
uniquely by the Fourier inversion formula. This completes the proof.
\end{proof}

Combining the two steps gives the desired uniqueness result for the original renormalized intensity data.

\begin{proof}[Proof of Theorem~\ref{thm:intensity_determines_rho}]
By Corollary~\ref{cor:equivalence_class_determines_Ak}, the equivalence class
$[\mathcal I_\rho]_\infty$ determines the scattering amplitudes
$A_k(\theta,\alpha)$ for all $k>k_\rho$ and all
$\theta,\alpha\in\mathbb S^{n-1}$ with $\theta\ne\alpha$.
By Lemma~\ref{lem:amplitude_determines_rho}, these scattering amplitudes
determine $\rho$ uniquely. Hence the renormalized far-field intensity data
determine $\rho$ uniquely.
\end{proof}

\subsection{Determination from phaseless superposition data}
We now prove the uniqueness result for the phaseless far-field data introduced
in Section~\ref{sec:formulation}. The key point is that measurements with
the four relative phases $a\in\{1,-1,i,-i\}$ determine the mixed products
of the scattering amplitudes
\begin{equation*}
    A_k(\theta,\alpha_2)\overline{A_k(\theta,\alpha_1)}.
\end{equation*}
The high-frequency asymptotics of $A_k$ then allow us to recover
$\widehat\rho(\xi)$ for every $\xi\in\mathbb R^n$.

\begin{proof}[Proof of Theorem~\ref{thm:phaseless_superposed}]
By the linearity relation \eqref{eq:superposition_amplitude}, we have
\begin{equation*}
    \left|A_k^{(a)}(\theta,\alpha_1,\alpha_2)\right|^2
    =
    |A_k(\theta,\alpha_1)|^2
    +
    |a|^2 |A_k(\theta,\alpha_2)|^2
    +
    2\Re\Bigl(
        a\,A_k(\theta,\alpha_2)\overline{A_k(\theta,\alpha_1)}
    \Bigr).
\end{equation*}
Since \(a\in\{1,-1,i,-i\}\), we obtain
\begin{equation}\label{eq:recover_real_part}
    \mathcal M_k(\alpha_1,\alpha_2,1,\theta)
-
\mathcal M_k(\alpha_1,\alpha_2,-1,\theta)
=
4\Re\Bigl(
A_k(\theta,\alpha_2)\overline{A_k(\theta,\alpha_1)}
\Bigr),
\end{equation}
\begin{equation}\label{eq:recover_imag_part}
    \mathcal M_k(\alpha_1,\alpha_2,i,\theta)
-
\mathcal M_k(\alpha_1,\alpha_2,-i,\theta)
=
-4\Im\Bigl(
A_k(\theta,\alpha_2)\overline{A_k(\theta,\alpha_1)}
\Bigr).
\end{equation}

Hence the data determine the product
\begin{equation*}
    A_k(\theta,\alpha_2)\overline{A_k(\theta,\alpha_1)}.
\end{equation*}

Next we use the high-frequency asymptotics of the scattering amplitude \eqref{eq:Ak_split} and \eqref{eq:Ak_split_bound_reminder} to obtain
\begin{multline}\label{eq:cross_asymptotics}
    A_k(\theta,\alpha_2)\overline{A_k(\theta,\alpha_1)}
    \\
    =
    \left|\frac{-2d_n g^2 k^{(n-1)/2}}{k-\Omega}\right|^2
    \widehat{\rho}\bigl(k(\theta-\alpha_2)\bigr)
    \overline{\widehat{\rho}\bigl(k(\theta-\alpha_1)\bigr)}
    +
    O\left(k^{n-4}\right),
\end{multline}
uniformly in \(\theta,\alpha_1,\alpha_2\in\mathbb{S}^{n-1}\). Taking $\alpha_1=\alpha_2=\theta$ for any $\theta\in\mathbb{S}^{n-1}$ gives
\begin{equation*}
    |A_k(\theta,\theta)|^2
    =
    \left|\frac{-2d_n g^2 k^{(n-1)/2}}{k-\Omega}\right|^2
    |\widehat{\rho}(0)|^2
    + O(k^{n-4}),
\end{equation*}
which is determined by the data through
\eqref{eq:recover_real_part}--\eqref{eq:recover_imag_part}. Hence the data
determine $|\widehat{\rho}(0)|^2$. Since \(\rho\ge 0\) and \(\rho\not\equiv 0\), we have $\widehat{\rho}(0)>0$. Therefore, the data determine $\widehat{\rho}(0)$.

We now fix \(\xi\in\mathbb{R}^n\). Choose \(k\) sufficiently large so that
\(|\xi|\le 2k\). We choose \(\theta_k,\alpha_{1,k}\in\mathbb{S}^{n-1}\) such that
\begin{equation}\label{eq:geometry_choice}
    k(\theta_k-\alpha_{1,k})=\xi.
\end{equation}
For example, choose \(\theta_k\in\mathbb{S}^{n-1}\) satisfying
\begin{equation*}
    \theta_k\cdot\xi = \frac{|\xi|^2}{2k},
\end{equation*}
and then set
\begin{equation*}
    \alpha_{1,k} = \theta_k-\frac{\xi}{k}.
\end{equation*}
A direct computation shows that \(|\alpha_{1,k}|=1\), hence
\(\alpha_{1,k}\in\mathbb{S}^{n-1}\), and \eqref{eq:geometry_choice} holds.

Next we choose $\alpha_{2,k}=\theta_k$. Then
\begin{equation*}
    k(\theta_k-\alpha_{2,k})=0.
\end{equation*}
Substituting these choices into \eqref{eq:cross_asymptotics}, we obtain
\begin{equation*}
    A_k(\theta_k,\alpha_{2,k})\overline{A_k(\theta_k,\alpha_{1,k})}
    =
    \left|\frac{-2d_n g^2 k^{(n-1)/2}}{k-\Omega}\right|^2
    \widehat{\rho}(0)\,\overline{\widehat{\rho}(\xi)}
    +
    O\left(k^{n-4}\right).
\end{equation*}

Equivalently,
\begin{equation*}
    \overline{\widehat{\rho}(\xi)}
    =
    \frac{1}{\widehat{\rho}(0)}
    \left|\frac{k-\Omega}{-2d_n g^2 k^{(n-1)/2}}\right|^2
    A_k(\theta_k,\alpha_{2,k})\overline{A_k(\theta_k,\alpha_{1,k})}
    +
    O(k^{-1}).
\end{equation*}
Since the right-hand side is determined by the data,
it follows that the data determine \(\widehat{\rho}(\xi)\) for every
\(\xi\in\mathbb{R}^n\) in the limit \(k\to\infty\).

Since $\xi\in\mathbb{R}^n$ was arbitrary, the data determine
$\widehat{\rho}$ on all of $\mathbb{R}^n$, and therefore determine
$\rho$ uniquely by the Fourier inversion formula.
\end{proof}

\appendix
\section{Auxiliary results}\label{sec:appendix}
In this appendix we prove Proposition~\ref{lem:GO_ptwise}.
We first establish two auxiliary results.

\begin{lemma}\label{lem:transport}
Let $\alpha\in\mathbb{S}^{n-1}$ and $h\in \mathscr{S}(\mathbb{R}^n)$.
For $\varepsilon>0$, define
\begin{equation*}
    T_\alpha^\varepsilon h(x)
  :=
  \mathcal{F}^{-1}\left[\frac{\widehat{h}(\eta)}{\alpha\cdot\eta-i\varepsilon}\right](x)
  = \frac{1}{(2\pi)^n}\int_{\mathbb{R}^n}
    \frac{e^{ix\cdot\eta}\widehat{h}(\eta)}{\alpha\cdot\eta-i\varepsilon}d\eta,
\end{equation*}
\begin{equation*}
    T_\alpha h(x)
    :=
    i\int_{-\infty}^{x\cdot\alpha} h(\tau\alpha+x_\perp)d\tau,
  \qquad x_\perp = x-(x\cdot\alpha)\alpha.
\end{equation*}
Then, for any compact set K,
\begin{equation*}
    \lim_{\varepsilon\to 0} \left\|T_{\alpha} h
    -
    T_{\alpha}^\varepsilon h 
    \right\|_{C^0(K)} =0.
\end{equation*}
\end{lemma}

\begin{proof}
We may assume $\alpha=e_1$ and write $x=(x_1,x')$, $\eta=(\eta_1,\eta')$.
Then
\begin{equation*}
    T_{e_1}^\varepsilon h(x)
  = \frac{1}{2\pi}\int_{\mathbb{R}}
    \frac{e^{ix_1\eta_1}}{\eta_1-i\varepsilon}
    \widetilde{h}(\eta_1,x') d\eta_1,
\end{equation*}
where
\begin{equation*}
    \widetilde{h}(\eta_1,x'):=\mathcal{F}_{\eta'}^{-1}[\widehat{h}(\eta_1,\cdot)](x')
=\mathcal{F}_{x_1}[h(\cdot,x')](\eta_1)
\end{equation*}
is the partial Fourier transform of $h$ in $x_1$ only. Since 
\begin{equation*}
    \mathcal{F}^{-1}_{\eta_1}[(\eta_1-i\varepsilon)^{-1}](x_1)
= i e^{-\varepsilon x_1} \mathbb{H}(x_1)
\end{equation*}
where $\mathbb{H}$ is the Heaviside function. Therefore, the convolution theorem implies
\begin{equation}\label{eq:expr_T_epsilon}
    T_{e_1}^\varepsilon h(x)
  = i\int_0^\infty e^{-\varepsilon s} h(x_1-s,x')ds
  = i\int_{-\infty}^{x_1} e^{-\varepsilon(x_1-\tau)} h(\tau,x')d\tau.
\end{equation}
Note that for $\tau \leq x_1$ and $x \in K$,
\begin{equation*}
    \left|e^{-\varepsilon(x_1 - \tau)} - 1\right|
    = 1 - e^{-\varepsilon(x_1 - \tau)}
    \leq
    \varepsilon(x_1 - \tau)
    \leq
    \varepsilon(C_K + |\tau|),
\end{equation*}
where $C_K = \sup_{x \in K} |x_1|$. Since $h \in \mathscr{S}(\mathbb{R}^n)$, for any $N \geq 3$ we have
\begin{equation*}
    |h(\tau, x')| \leq \|h\|_{\mathscr{S}_N}\,
\langle(\tau, x')\rangle^{-N}.
\end{equation*}
Therefore,
\begin{multline*}
    \left\|T_{e_1}^\varepsilon h
    - T_{e_1} h\right\|_{C^0(K)} \leq \sup_{x \in K}  \int_{-\infty}^{x_1} \left|e^{-\varepsilon(x_1-\tau)} - 1\right||h(\tau,x')|d\tau\\
    \leq
    \varepsilon \|h\|_{\mathscr{S}_N}
    \int_{\mathbb{R}}
    (C_K + |\tau|) \langle\tau\rangle^{-N} d\tau
    \leq
    C_{K,N} \varepsilon \|h\|_{\mathscr{S}_N},
\end{multline*}
where the integral converges since $N \geq 3$. Hence
\begin{equation*}
    \left\|T_{e_1}^\varepsilon h
    - T_{e_1} h\right\|_{C^0(K)}
    \to 0
    \qquad \text{as } \varepsilon \to 0^+.
\end{equation*}
This completes the proof.
\end{proof}

\begin{lemma}\label{lem:aux_sing}
Let $k> 1$.  Define
\begin{equation*}
    \sigma_k(\eta):=\eta_1+\frac{|\eta|^2}{2k},
  \qquad
  a_k(\eta):=\frac{|ke_1+\eta|+k}{2k},
\end{equation*}
and for $\theta,\vartheta\in[0,1]$ set
\begin{equation}\label{eq:aux_interp}
  \sigma_k^\vartheta(\eta):=\eta_1+\frac{\vartheta|\eta|^2}{2k},
  \qquad
  c_\theta(\eta):=\theta a_k(\eta)+(1-\theta).
\end{equation}
For $\varepsilon>0$, set
\begin{equation}\label{eq:aux_X}
  X_{\theta,\vartheta}(\eta)
  :=\sigma_k^\vartheta(\eta)-i\varepsilon c_\theta(\eta).
\end{equation}
Then there exist constant $C>0$, depending only on $n$, such that for all $k>1$,
$\theta,\vartheta\in[0,1]$, $\varepsilon\in(0,1]$,
and all $\psi\in\mathscr{S}(\mathbb{R}^n)$ supported in $\{|\eta|\leq k/2\}$:
\begin{align}
  \left|\int_{\mathbb{R}^n}\frac{\psi(\eta)}{X_{\theta,\vartheta}(\eta)}\,d\eta\right|
  &\leq C\|\psi\|_{\mathscr{S}_{n+4}},
  \label{eq:aux_inv1}\\
  \left|\int_{\mathbb{R}^n}\frac{\psi(\eta)}{X_{\theta,\vartheta}(\eta)^2}\,d\eta\right|
  &\leq C\|\psi\|_{\mathscr{S}_{n+5}}.
  \label{eq:aux_inv2}
\end{align}
\end{lemma}

\begin{proof}
We write $X=X_{\theta,\vartheta}$. Next, we list certain properties of functions $\sigma_k^\vartheta$, $c_\theta$, and $X$ on $\{|\eta|\leq k/2\}$.

\textbf{Properties of $\sigma_k^\vartheta$.}
We compute
\begin{equation*}
    \partial_{\eta_1}\sigma_k^\vartheta = 1+ \frac{\vartheta\eta_1}{k},
\end{equation*}
so
\begin{equation}\label{eq:aux_sigma_bounds}
  \frac{1}{2}\leq \partial_{\eta_1}\sigma_k^\vartheta\le\frac{3}{2}
  \qquad
  \text{on }\{|\eta_1|\leq k/2\}.
\end{equation}
Moreover,
\begin{equation*}
    \partial_{\eta_i}\partial_{\eta_j}\sigma_k^\vartheta = \frac{\vartheta\delta_{ij}}{k} \leq \frac{1}{k}
\end{equation*}
and all derivatives of order greater than $2$ are zero. Hence,
\begin{equation}\label{eq:aux_sigma_allderivs}
  |\partial_\eta^\nu\sigma_k^\vartheta|\leq C_\nu
  \qquad\text{on }\{|\eta|\leq k/2\},\quad\text{for }|\nu|\geq 1,
\end{equation}
uniformly in $k> 1$ and $\vartheta\in[0,1]$.

For $\vartheta=0$, $\sigma_k^0=\eta_1$ has the unique zero $\eta_1^*=0$.
For $\vartheta>0$, the equation $\sigma_k^\vartheta(\eta_1,\eta')=0$
is quadratic in $\eta_1$ with two roots:
\begin{equation*}
    \eta_1^*=\frac{k}{\vartheta}\Bigl(-1+\sqrt{1-\tfrac{\vartheta^2|\eta'|^2}{k^2}}\Bigr),
  \qquad
  \eta_1^{**}=\frac{k}{\vartheta}\Bigl(-1-\sqrt{1-\tfrac{\vartheta^2|\eta'|^2}{k^2}}\Bigr).
\end{equation*}

Since $|\eta_1^{**}|\geq k/\vartheta\geq k$, we obtain $\eta_1^{**} \notin \{|\eta_1|\leq k/2\}$. For $\eta' \in \{|\eta'|\leq k/2\}$, we estimate
\begin{equation}\label{eq:aux_eta1star}
    |\eta_1^*|=\frac{k}{\vartheta}\Bigl(1-\sqrt{1-\tfrac{\vartheta^2|\eta'|^2}{k^2}}\Bigr) \leq \frac{\frac{\vartheta|\eta'|^2}{k}}{1+\sqrt{1-\tfrac{\vartheta^2|\eta'|^2}{k^2}}} \leq \frac{|\eta'|^2}{k} \leq \frac{k}{4} < \frac{k}{2}.
\end{equation}
Hence, $\eta_1^* \in \{|\eta_1|\leq k/2\}$. Using the mean value theorem and \eqref{eq:aux_sigma_bounds}, we obtain the two-sided bound:
\begin{equation}\label{eq:aux_sigma_two_sided}
  \tfrac{1}{2}|\eta_1-\eta_1^*|
  \le|\sigma_k^\vartheta(\eta_1,\eta')|
  \le\tfrac{3}{2}|\eta_1-\eta_1^*|,
  \qquad
  \text{for } \eta \in \left\{|\eta|<k/2\right\}.
\end{equation}
Note that in general $(\eta_1^*,\eta')$ is not in $\left\{|\eta|<k/2\right\}$.

\textbf{Properties of $c_\theta$.}
Note that 
\begin{equation*}
    \frac{k}{2}\leq |ke_1+\eta|\leq \frac{3k}{2},
    \qquad
    \text{on }
    \{|\eta|\leq k/2\}
\end{equation*}
Therefore, 
\begin{equation}\label{eq:aux_c_bounds}
  \frac{3}{4}\leq c_\theta\le\frac{5}{4}
  \qquad\text{on }\{|\eta|\leq k/2\}.
\end{equation}
Writing $\xi=ke_1+\eta$ with $|\xi|\geq k/2$, we have
$\partial_{\eta_j}a_k=\xi_j/(2k|\xi|)$, and each further $\eta$-derivative
brings a factor of $|\xi|^{-1}\leq 2/k$.
Hence
\begin{equation}\label{eq:aux_a_allderivs}
    |\partial_\eta^\nu a_k|\leq C_\nu k^{-|\nu|},
    \qquad
    \text{for }
    |\nu|\geq 1
\end{equation}
and therefore
\begin{equation}\label{eq:aux_c_allderivs}
  |\partial_\eta^\nu c_\theta|\leq C_\nu k^{-|\nu|}\leq C_\nu
  \qquad\text{on }\{|\eta|\leq k/2\},\quad|\nu|\geq 1,
\end{equation}
uniformly in $k> 1$ and $\theta\in[0,1]$.

\textbf{Properties of $X$.}
Due to \eqref{eq:aux_c_bounds},
\begin{equation}\label{eq:aux_Xnonzero}
  |X|= |\sigma_k^\vartheta-i\varepsilon c_\theta| \geq \varepsilon c_\theta\geq \frac{3}{4}\varepsilon>0,
  \qquad
  \text{on } \{|\eta|\leq k/2\},
\end{equation}
and hence, $X$ does not vanish on $\{|\eta|\leq k/2\}$.
Also $\operatorname{Im} X=-\varepsilon c_\theta<0$, so $X$ lies in the open
lower half-plane. By \eqref{eq:aux_sigma_bounds}
\begin{equation}\label{eq:aux_dX_lower}
    |\partial_{\eta_1}X|
    =\sqrt{(\partial_{\eta_1}\sigma_k^\vartheta)^2
    +\varepsilon^2(\partial_{\eta_1}c_\theta)^2}
    \geq |\partial_{\eta_1}\sigma_k^\vartheta|
    \geq\frac{1}{2}.
\end{equation}
Estimates \eqref{eq:aux_sigma_allderivs} and \eqref{eq:aux_c_allderivs} imply that
\begin{equation}\label{eq:aux_X_allderivs}
  |\partial_\eta^\nu X|\leq C_\nu
  \qquad\text{on }\{|\eta|\leq k/2\},\quad\text{for }|\nu|\geq 1,
\end{equation}
uniformly in $k> 1$, $\theta,\vartheta\in[0,1]$, and $\varepsilon\in(0,1]$.

Combining the last two estimates gives 
\begin{equation}\label{eq:aux_inv_dX_allderivs}
  \left|\partial_\eta^\nu\left[\frac{1}{\partial_{\eta_1}X}\right]\right|\leq C_\nu
  \qquad\text{on }\{|\eta|\leq k/2\},\quad\text{for all }\nu.
\end{equation}

\textbf{Proof of \eqref{eq:aux_inv1}.}
As we established, $X$ and $\partial_{\eta_1} X$ do not vanish on $\operatorname{supp} (\psi) \subset \{|\eta|\leq k/2\}$. Therefore, by using
\begin{equation*}
    \frac{1}{X} = \frac{\partial_{\eta_1}[\log X]}{\partial_{\eta_1} X}
\end{equation*}
and integrating by parts, we obtain
\begin{equation}\label{eq:aux_IBP}
  \int_{\mathbb{R}^n}\frac{\psi}{X}\,d\eta
  = -\int_{\mathbb{R}^n}\log X(\eta)\;\partial_{\eta_1}\!\left[
      \frac{\psi(\eta)}{\partial_{\eta_1}X(\eta)}
    \right]d\eta.
\end{equation}
Here $\log X$ is the principal-branch logarithm, well-defined since
$\operatorname{Im} X<0$ (so $\arg X\in(-\pi,0)$). Therefore,
\begin{equation}\label{eq:logX}
    |\log X| \leq |\log |X|| + |\arg X| \leq |\log |X|| + \pi.
\end{equation}

By the definition of $X$ and \eqref{eq:aux_c_bounds},
\begin{equation*}
    |\sigma_k^\vartheta| \leq |X| \leq |\sigma_k^\vartheta|  + \frac{5}{4}.
\end{equation*}
For $|X|\geq 1$, we estimate
\begin{equation*}
    \log |X| \leq \log \left(|\sigma_k^\vartheta|  + \frac{5}{4}\right) \leq 
    \begin{cases}
        \log\left(\frac{9}{4}|\sigma_k^\vartheta| \right) & |\sigma_k^\vartheta| \geq 1,\\
        \log\left(\frac{9}{4} \right) & |\sigma_k^\vartheta| \leq 1.\
    \end{cases}
\end{equation*}
For $|X|\leq 1$,
\begin{equation*}
    |\log |X|| = - \log |X|\leq - \log |\sigma_k^\vartheta| = |\log |\sigma_k^\vartheta||.
\end{equation*}
In both cases,
\begin{equation*}
    |\log |X|| \leq |\log |\sigma_k^\vartheta|| + \log\left(\frac{9}{4} \right).
\end{equation*}
Combining this with \eqref{eq:aux_sigma_two_sided} and \eqref{eq:logX}, we obtain
\begin{equation}\label{eq:aux_logX}
  |\log X(\eta)|\le|\log|\eta_1-\eta_1^*(\eta')||+C
\end{equation}
for some absolute constant $C$.

Next, we estimate the second term in the integral \eqref{eq:aux_IBP}. Set $\varphi(\eta):=\partial_{\eta_1}[\psi/\partial_{\eta_1}X]$.
Since $\psi$ is smooth and supported in $\{|\eta|\leq k/2\}$, \eqref{eq:aux_inv_dX_allderivs} implies that
\begin{equation}\label{eq:aux_phi_bound}
  \|\varphi\|_{\mathscr{S}_m}\leq C\|\psi\|_{\mathscr{S}_{m+1}}
\end{equation}
for every $m$, with $C=C(m,n)$.

Combining \eqref{eq:aux_IBP} and \eqref{eq:aux_logX} gives
\begin{equation*}
    \left|\int_{\mathbb{R}^n}\frac{\psi}{X} d\eta\right|
    \leq C\int_{\mathbb{R}^n}\left(1+|\log|\eta_1-\eta_1^*(\eta')||\right) |\varphi(\eta)| d\eta.
\end{equation*}
We split the inner $\eta_1$-integral at $|\eta_1-\eta_1^*|=1$:

\begin{multline}\label{eq:spl_int}
    \left|\int_{\mathbb{R}^n}\frac{\psi}{X} d\eta\right|
  \leq 
  C\int_{\mathbb{R}^n} |\varphi(\eta)| d\eta
  + C \int_{\mathbb{R}^{n-1}}\int_{|\eta_1 - \eta_1^*(\eta')| \leq 1} |\log|\eta_1-\eta_1^*(\eta')| |\varphi(\eta)| d\eta\\
  + C \int_{\mathbb{R}^{n-1}}\int_{|\eta_1 - \eta_1^*(\eta')| \geq 1} |\log|\eta_1-\eta_1^*(\eta')| |\varphi(\eta)| d\eta.
\end{multline}
Since $\varphi \in C_c^\infty(\mathbb{R}^n)$,
\begin{equation*}
    |\varphi(\eta)|\le\|\varphi\|_{\mathscr{S}_{p+1}}(1+|\eta|)^{-p}
    \qquad
    \text{for any } p\in \mathbb{N}.
\end{equation*}
Choosing $p = n + 2$, we estimate the first term of the right-hand side of \eqref{eq:spl_int}
\begin{equation*}
    \int_{\mathbb{R}^n} |\varphi(\eta)| d\eta \leq C \|\varphi\|_{\mathscr{S}_{n+3}}.
\end{equation*}
To estimate the second term, we note that
\begin{equation*}
    \int_{|\eta_1-\eta_1^*|\leq 1}\frac{|\log|\eta_1-\eta_1^*||}{(1+|\eta|)^p} d\eta_1
    \leq \frac{1}{(1+|\eta'|)^p}\int_{|s|\leq 1}|\log|s|| ds
    = \frac{2}{(1+|\eta'|)^p}.
\end{equation*}
Hence,
\begin{equation*}
    \int_{\mathbb{R}^{n-1}}\int_{|\eta_1 - \eta_1^*(\eta')| \leq 1} |\log|\eta_1-\eta_1^*(\eta')| |\varphi(\eta)| d\eta \leq  C \|\varphi\|_{\mathscr{S}_{n+3}}.
\end{equation*}
To estimate the last term of the of the right-hand side of \eqref{eq:spl_int}, by using \eqref{eq:aux_eta1star} and $\log|t|\le|t|$ for $|t|>1$, we obtain
\begin{equation*}
    \frac{|\log|\eta_1-\eta_1^*||}{(1+|\eta|)^p}
    \leq 
    \frac{|\eta_1-\eta_1^*|}{(1+|\eta|)^p}
    \leq
    \frac{|\eta_1|+|\eta_1^*|}{(1+|\eta|)^p}
    \leq
    \frac{C(1+|\eta|)}{(1+|\eta|)^p}=\frac{C}{(1+|\eta|)^{p-1}}.
\end{equation*}
Therefore,
\begin{equation*}
    \int_{\mathbb{R}^{n-1}}\int_{|\eta_1 - \eta_1^*(\eta')| \geq 1} |\log|\eta_1-\eta_1^*(\eta')| |\varphi(\eta)| d\eta \leq  C \|\varphi\|_{\mathscr{S}_{n+3}}.
\end{equation*}
By combining estimates for the three terms of the right-hand side of \eqref{eq:spl_int}, and then using \eqref{eq:aux_phi_bound}, 
\begin{equation*}
    \left|\int_{\mathbb{R}^n}\frac{\psi}{X}d\eta\right|
    \leq C\|\varphi\|_{\mathscr{S}_{n+3}}
    \leq C\|\psi\|_{\mathscr{S}_{n+4}}
\end{equation*}
This proves \eqref{eq:aux_inv1}.

\textbf{Proof of \eqref{eq:aux_inv2}.}
Since $X$ and $\partial_{\eta_1} X$ do not vanish on $\operatorname{supp} (\psi) \subset \{|\eta|\leq k/2\}$, we can write 
\begin{equation*}
    \frac{1}{X^2} = - \frac{1}{\partial_{\eta_1}X} \partial_{\eta_1} \left[\frac{1}{X}\right]
\end{equation*}
on $\{|\eta|\leq k/2\}$. Then, the integration by parts gives
\begin{equation*}
    \int_{\mathbb{R}^n}\frac{\psi}{X^2}d\eta
    = \int_{\mathbb{R}^n}\frac{1}{X}
    \partial_{\eta_1}\left[\frac{\psi}{\partial_{\eta_1}X}\right]d\eta = \int_{\mathbb{R}^n}\frac{\varphi}{X}
    d\eta.
\end{equation*}
 Recall that $\varphi$ is a smooth function supported in $\{|\eta|\leq k/2\}$. Therefore, applying \eqref{eq:aux_inv1} and \eqref{eq:aux_phi_bound}, we obtain 
\begin{equation*}
    \left|\int_{\mathbb{R}^n}\frac{\psi}{X^2} d\eta\right| 
    = \left|\int_{\mathbb{R}^n}\frac{\varphi}{X} d\eta\right| 
    \leq C\|\varphi\|_{\mathscr{S}_{n+4}} 
    \leq C\|\psi\|_{\mathscr{S}_{n+5}}.
\end{equation*}
This gives \eqref{eq:aux_inv2}.
\end{proof}

\subsection{Proof of Proposition~\ref{lem:GO_ptwise}}

\begin{proof}
Since $\partial_x^\beta$ corresponds to multiplication
by $(i\eta)^\beta$ in frequency, it suffices to prove the estimate
for $\beta=0$, replacing $h$ by $\partial^\beta h$.
We may assume $\alpha=e_1$ and write
$\eta=(\eta_1,\eta')\in\mathbb{R}\times\mathbb{R}^{n-1}$.
We use the notation $\sigma_k$, $a_k$, $\sigma_k^\vartheta$, $c_\theta$,
$X_{\theta,\vartheta}$ from Lemma~\ref{lem:aux_sing}.

Set
\begin{equation*}
    \phi_x(\eta):=e^{ix\cdot\eta}\widehat{h}(\eta).
\end{equation*}
Next, fix $\chi_0\in C_c^\infty(\mathbb{R}^n)$ such that
\begin{equation*}
    \chi_0=1 \quad \text{on } \{|\eta|\leq 1\},
\qquad
\chi_0=0 \quad \text{on } \{|\eta|\geq 2\},
\end{equation*}
and define
\begin{equation*}
    \chi_k(\eta):=\chi_0(4\eta/k).
\end{equation*}
Then
\begin{equation*}
    \chi_k=1 \quad \text{on } \{|\eta|\leq k/4\},
\qquad
\chi_k=0 \quad \text{on } \{|\eta|\geq k/2\}.
\end{equation*}

We decompose
\begin{equation*}
    \phi_x=\phi_x^{\mathrm{in}}+\phi_x^{\mathrm{out}},
\qquad
\phi_x^{\mathrm{in}}:=\chi_k\phi_x,
\qquad
\phi_x^{\mathrm{out}}:=(1-\chi_k)\phi_x.
\end{equation*}
We also define
\begin{equation*}
    h^{\mathrm{in}}:=\mathcal{F}^{-1}[\chi_k\widehat h],
\qquad
h^{\mathrm{out}}:=\mathcal{F}^{-1}[(1-\chi_k)\widehat h].
\end{equation*}
Then
\begin{equation*}
    h=h^{\mathrm{in}}+h^{\mathrm{out}},
\end{equation*}
and
\begin{equation*}
    \phi_x^{\mathrm{in}}(\eta)=e^{ix\cdot\eta}\widehat{h^{\mathrm{in}}}(\eta),
\qquad
\phi_x^{\mathrm{out}}(\eta)=e^{ix\cdot\eta}\widehat{h^{\mathrm{out}}}(\eta).
\end{equation*}
Next we decompose
\begin{equation*}
    \widetilde{R}_k h(x)-T_{e_1} h(x)
=
(\widetilde{R}_k h^{\mathrm{in}}(x)-T_{e_1} h^{\mathrm{in}}(x))
+
(\widetilde{R}_k h^{\mathrm{out}}(x)-T_{e_1} h^{\mathrm{out}}(x)),
\end{equation*}
and estimate the two terms on the right-hand side separately.

\textbf{Exterior contribution.}
Next we estimate 
\begin{equation*}
    \widetilde{R}_k h^{\mathrm{out}}(x) - T_{e_1} h^{\mathrm{out}}(x)
\end{equation*}
We now estimate $\widetilde{R}_k h^{\mathrm{out}}$.
By \eqref{eq:conj_identity} with $\varepsilon=0$,
\begin{equation*}
    \widetilde{R}_k h^{\mathrm{out}}(x)
=
e^{-ik e_1 \cdot x}\,
R_k^0\!\bigl(e^{ik e_1 \cdot (\cdot)} h^{\mathrm{out}}\bigr)(x).
\end{equation*}
We set $f := e^{ik e_1 \cdot (\cdot)} h^{\mathrm{out}}$, $l=[n/2]+1$, and $\delta=3/4$. By Corollary~\ref{cor:free_frac_Hl}, for $k\geq1$, we estimate
\begin{multline}\label{eq:Rk_out_Hl}
    \|\widetilde{R}_k h^{\mathrm{out}}\|_{H^l_{-\delta}}
    =
    \|e^{-ik e_1 \cdot x} R_k^0 f\|_{H^l_{-\delta}}
    \leq
    C_l k^l \|R_k^0 f\|_{H^l_{-\delta}}
    \\
    \leq
    C_l k^l \|f\|_{H^l_\delta}
    =
    C_l k^l
    \|e^{ik e_1 \cdot (\cdot)} h^{\mathrm{out}}\|_{H^l_\delta}
    \leq
    C_l k^{2l} \|h^{\mathrm{out}}\|_{H^l_\delta}.
\end{multline}
Let us bound $\|h^{\mathrm{out}}\|_{H^l_\delta}$. Since $(1-\Delta)^{l/2} h^{\mathrm{out}} \in \mathscr{S}(\mathbb{R}^n)$, for an integer $N>n/2 + \delta$, we estimate 
\begin{multline*}\label{eq:hout_Hldelta_start}
    \|h^{\mathrm{out}}\|_{H^l_\delta}^2
    =
    \int_{\mathbb{R}^n}
    \langle x \rangle^{2\delta}
    |(1-\Delta)^{l/2} h^{\mathrm{out}}(x)|^2 dx 
    \leq
    \int_{\mathbb{R}^n}
    \langle x \rangle^{2\delta} \langle x \rangle^{-2N}
    \|(1-\Delta)^{l/2} h^{\mathrm{out}}\|_{\mathscr{S}_N}^2 dx\\
    \leq C_{\delta,N} \|(1-\Delta)^{l/2} h^{\mathrm{out}}\|_{\mathscr{S}_N}^2.
\end{multline*}
Since $\mathcal{F}:\mathscr{S}\to\mathscr{S}$ is a topological isomorphism, there exist $C_N>0$ and $M\in \mathbb{N}$ such that 
\begin{multline*}
    \|h^{\mathrm{out}}\|_{H^l_\delta} \leq C_{\delta,N}\|(1-\Delta)^{l/2} h^{\mathrm{out}}\|_{\mathscr{S}_N} \leq C_{\delta,N}C_N \|\mathcal{F}\left[(1-\Delta)^{l/2} h^{\mathrm{out}}\right]\|_{\mathscr{S}_M} \\
    \leq C_{\delta,N} \|\langle\cdot\rangle^l (1-\chi_k(\cdot))\widehat{h}(\cdot)\|_{\mathscr{S}_M}.
\end{multline*}
Since $1 - \chi_k$ is supported in $\{|\eta|\geq k/4\}$ and $\langle\eta\rangle^{-1} \leq C/k$ on $\{|\eta|\geq k/4\}$, we conclude that 
\begin{equation*}
    \|h^{\mathrm{out}}\|_{H^l_\delta} \leq \frac{C_{\delta,N,L}}{k^L} \|\langle\cdot\rangle^{l+L} (1-\chi_k(\cdot))\widehat{h}(\cdot)\|_{\mathscr{S}_M}.
\end{equation*}
By the definition of $\chi_k$,
\begin{equation*}
    \|h^{\mathrm{out}}\|_{H^l_\delta} \leq \frac{C_{\delta,N,L}}{k^L} \|\widehat{h}\|_{\mathscr{S}_{M + l + L}}.
\end{equation*}
We put the last estimate into \eqref{eq:Rk_out_Hl} to obtain 
\begin{equation}\label{eq:Rh_leq_hath}
    \|\widetilde{R}_k h^{\mathrm{out}}\|_{H^l_{-\delta}} \leq \frac{C_{\delta,l}}{k^{L-2l}} \|\widehat{h}\|_{\mathscr{S}_{M + l + L}}.
\end{equation}
Since $\langle x\rangle^{-\delta}$ is bounded below on $K$,
the norms $H^l_{-\delta}(K)$ and $H^l(K)$ are equivalent. Therefore, by choosing $L=2l+1$, and using the Sobolev embedding theorem, we obtain
\begin{equation}\label{eq:Rk_out_ptwise}
    \sup_{x\in K} |\widetilde{R}_k h^{\mathrm{out}}(x)| \leq \frac{C_{K}}{k} \|\widehat{h}\|_{\mathscr{S}_{M}}.
\end{equation}
for some $M\in \mathbb{N}$.

Next we estimate $T_{e_1} h^{\mathrm{out}}$.
By Lemma~\ref{lem:transport},
\begin{equation*}
    T_{e_1} h^{\mathrm{out}}(x)
    =
    i \int_{-\infty}^{x_1} h^{\mathrm{out}}(\tau, x') d\tau.
\end{equation*}
Then, we estimate
\begin{equation}\label{T_e1_h_out}
    |T_{e_1}h^{\mathrm{out}}(x)|
    \leq
    \int_{\mathbb R}|h^{\mathrm{out}}(\tau,x')|d\tau 
    \leq
    \sup_{y\in\mathbb R^n}
    \langle y\rangle^2 |h^{\mathrm{out}}(y)|
    \int_{\mathbb R}
    \langle(\tau,x')\rangle^{-2}d\tau 
    \leq
    C\|h^{\mathrm{out}}\|_{\mathscr S_2}.
\end{equation}
Let
\begin{equation*}
    F_k(\eta) := (1-\chi_k(\eta))\widehat h(\eta),
    \qquad
    h^{\mathrm{out}} = \mathcal F^{-1}F_k .
\end{equation*}
For all multi-indices $\alpha,\beta$ with
$|\alpha|,|\beta|\leq 2$, we have
\begin{equation*}
    x^\alpha\partial_x^\beta h^{\mathrm{out}}(x)
    =
    (2\pi)^{-n}
    \int_{\mathbb R^n}
    e^{ix\cdot\eta}
    i^{-|\alpha|}
    \partial_\eta^\alpha\!\left[(i\eta)^\beta F_k(\eta)\right]
    \,d\eta .
\end{equation*}
Therefore,
\begin{equation*}
    \|h^{\mathrm{out}}\|_{\mathscr S_2}
    \leq
    C
    \sum_{|\alpha|,|\beta|\leq 2}
    \left\|
    \partial_\eta^\alpha\!\left[\eta^\beta(1-\chi_k)\widehat h\right]
    \right\|_{L^1(\mathbb R^n)} .
\end{equation*}
Since $1-\chi_k$ is supported in $\{|\eta|\geq k/4\}$, we have
\begin{equation}\label{supp_est}
    1\leq \frac{4|\eta|}{k}
    \leq \frac{4\langle\eta\rangle}{k}
    \qquad
    \text{on } \operatorname{supp}(1-\chi_k).
\end{equation}
Also, for every multi-index $\gamma$,
\begin{equation*}
    |\partial_\eta^\gamma \chi_k(\eta)|
    \leq
    C_\gamma k^{-|\gamma|}
    \leq C_\gamma,
    \qquad k\geq 1.
\end{equation*}
Using Leibniz' rule, we obtain, for $|\alpha|,|\beta|\leq 2$,
\begin{equation*}
    \left|
    \partial_\eta^\alpha\!\left[\eta^\beta(1-\chi_k)\widehat h\right]
    \right|
    \leq
    C
    \sum_{|\mu|\leq N_0}
    \langle\eta\rangle^{N_0}
    \left|\partial_\eta^\mu \widehat h(\eta)\right|
\end{equation*}
for some fixed integer $N_0$. Since $1-\chi_k$ and all its
derivatives are supported where $|\eta|\geq k/4$, \eqref{supp_est} implies
\begin{equation*}
    \langle\eta\rangle^{N_0}
    \leq
    \frac{C}{k}\langle\eta\rangle^{N_0+1}
    \qquad
    \text{on } \operatorname{supp}(1-\chi_k).
\end{equation*}
Thus, after increasing the Schwartz seminorm index, we get
\begin{equation*}
    \left\|
    \partial_\eta^\alpha\!\left[\eta^\beta(1-\chi_k)\widehat h\right]
    \right\|_{L^1(\mathbb R^n)}
    \leq
    \frac{C}{k}
    \|\widehat h\|_{\mathscr S_N},
    \qquad
    |\alpha|,|\beta|\leq 2.
\end{equation*}
Consequently,
\begin{equation*}
    \|h^{\mathrm{out}}\|_{\mathscr S_2}
    \leq
    \frac{C}{k}\|\widehat h\|_{\mathscr S_N},
\end{equation*}
and hence, \eqref{T_e1_h_out} gives
\begin{equation*}
\sup_{x \in K}
|T_{e_1} h^{\mathrm{out}}(x)|
\leq
\frac{C}{k} \|\widehat{h}\|_{\mathscr{S}_{M'}}.
\end{equation*}
Combining this with \eqref{eq:Rk_out_ptwise}, we conclude that
\begin{equation}\label{eq:exterior_final}
\sup_{x \in K}
\left|
\widetilde{R}_k h^{\mathrm{out}}(x)
-
T_{e_1} h^{\mathrm{out}}(x)
\right|
\le
\frac{C}{k}\,
\|\widehat{h}\|_{\mathscr{S}_{M'}}.
\end{equation}

\textbf{Interior contribution.} From the identity
\begin{equation*}
    (|ke_1+\eta|-k)(|ke_1+\eta|+k) = 2k\eta_1+|\eta|^2,
\end{equation*}
we obtain 
\begin{equation*}
    |ke_1+\eta|-k = \frac{2k\eta_1+|\eta|^2}{|ke_1+\eta|+k} = \frac{2k}{|ke_1+\eta|+k} \cdot \left( \eta_1 + \frac{|\eta|^2}{2k} \right) = \frac{\sigma_k(\eta)}{a_k(\eta)}
\end{equation*}
Therefore, for every $\varepsilon>0$,
\begin{equation}\label{eq:exact_factor}
  \frac{1}{|ke_1+\eta|-k-i\varepsilon}
  = \frac{c_1(\eta)}{X_{1,1}(\eta)}.
\end{equation}
Then, we write
\begin{align*}
   \widetilde{R}_{k,\varepsilon} h^{\mathrm{in}}(x)-T_{e_1}^\varepsilon h^{\mathrm{in}}(x)
  & =  \frac{1}{(2\pi)^n}\int_{\mathbb{R}^n} \left[   \frac{1}{|ke_1+\eta|-k-i\varepsilon} - \frac{1}{\eta_1 - i \varepsilon} \right]e^{ix\cdot\eta}\widehat{h^{\mathrm{in}}}(\eta) d\eta\\
    & = \frac{1}{(2\pi)^n}\int_{\mathbb{R}^n}
    \left[\frac{c_1(\eta)}{X_{1,1}(\eta)}
    -\frac{1}{X_{0,0}(\eta)}\right]\phi_x^{\mathrm{in}}(\eta) d\eta
\end{align*}
We decompose the bracket as $E_1+E_2+E_3$:
\begin{align*}
  E_1(\eta) &:= \frac{c_1(\eta)-1}{X_{1,1}(\eta)},\\
  E_2(\eta) &:= \frac{1}{X_{1,1}(\eta)} - \frac{1}{X_{0,1}(\eta)},\\
  E_3(\eta) &:= \frac{1}{X_{0,1}(\eta)} - \frac{1}{X_{0,0}(\eta)}.
\end{align*}
We estimate the pairing $\langle E_i, \phi_x^{\mathrm{in}} \rangle$
for $i = 1, 2, 3$, where
$\operatorname{supp}(\phi_x^{\mathrm{in}}) \subset \{|\eta| \leq k/2\}$.

Note that $c_1 = a_k$ by definition. Therefore, 
\begin{equation*}
    \int_{\mathbb{R}^n} E_1(\eta) \phi_x^{\mathrm{in}}(\eta) d\eta
    =
    \int_{\mathbb{R}^n} \frac{(a_k(\eta) - 1) \phi_x^{\mathrm{in}}(\eta)}{X_{1,1}(\eta)} d\eta.
\end{equation*}
Since $\phi_x^{\mathrm{in}}$ is smooth and supported in $\{|\eta| \leq k/2\}$, Lemma \ref{lem:aux_sing} implies that there exists $M \in \mathbb{N}$ such that
\begin{equation}\label{eq:E1_vs_aphi}
    \left|\int_{\mathbb{R}^n} E_1(\eta) \phi_x^{\mathrm{in}}(\eta) d\eta \right| 
    \leq 
    C \|(a_k - 1)\phi_x^{\mathrm{in}}\|_{\mathscr{S}_M}
\end{equation}
uniformly in $\varepsilon\in (0,1]$ and $k\geq 1$. Combining \eqref{eq:aux_a_allderivs} with
\begin{equation*}
    |a_k(\eta) - 1|
    =
    \frac{||ke_1 + \eta| - k|}{2k}
    \leq
    \frac{|\eta|}{2k}
    \qquad
    \text{on }
    \{|\eta| \leq k/2\},
\end{equation*}
we conclude that 
\begin{equation}\label{eq:a_k_phi_leq_phi}
    \|(a_k - 1) \phi_x^{\mathrm{in}}\|_{\mathscr{S}_M} \leq \frac{C}{k} \| \phi_x^{\mathrm{in}}\|_{\mathscr{S}_{M+1}}.
\end{equation}
Recall that
\begin{equation*}
    \phi_x^{\mathrm{in}}(\eta) = e^{ix\cdot\eta}\chi_k(\eta)\widehat h(\eta)
\end{equation*}
Since $|\partial_\eta^\nu \chi_k| < C/k^{|\nu|}$ and the polynomials, in $x$ variables, are bounded in $K$, we conclude that
\begin{equation}\label{eq:phi_in_vs_hat_h}
    \sup_{x \in K} \| \phi_x^{\mathrm{in}}\|_{\mathscr{S}_{M+1}} \leq C \|\widehat h\|_{\mathscr{S}_{M+1}} 
\end{equation}
Combining \eqref{eq:E1_vs_aphi}, \eqref{eq:a_k_phi_leq_phi}, and \eqref{eq:phi_in_vs_hat_h} gives
\begin{equation}\label{eq:E1}
     \sup_{x \in K}\left|\int_{\mathbb{R}^n} E_1(\eta) \phi_x^{\mathrm{in}}(\eta) d\eta \right| \leq \frac{C}{k} \|\widehat h\|_{\mathscr{S}_{M+1}} 
\end{equation}
uniformly in $\varepsilon \in (0,1]$ and $k\geq 1$.

Next, we estimate 
\begin{equation*}
    \langle E_2, \phi_x^{\mathrm{in}} \rangle = \int_{\mathbb{R}^n}
    E_2(\eta) \phi_x^{\mathrm{in}}(\eta) d\eta 
    = 
    \int_{\mathbb{R}^n}
    \left[\frac{1}{X_{1,1}(\eta)} - \frac{1}{X_{0,1}(\eta)}\right]\phi_x^{\mathrm{in}}(\eta) d\eta.
\end{equation*}
We use the integral identity
\begin{equation}\label{eq:resolvent_identity_param}
\frac{1}{A} - \frac{1}{B}
=
(B - A) \int_0^1 \frac{d\theta}{[\theta A + (1-\theta)B]^2}
\end{equation}
Take $A = X_{1,1}$ and $B = X_{0,1}$. Note that
\begin{equation*}
    B - A =  i\varepsilon (a_k - 1),
    \qquad
    \theta X_{1,1} + (1-\theta) X_{0,1} = X_{\theta,1},
\end{equation*}
for $\theta \in [0,1]$. Therefore, 
\begin{equation}\label{eq:E2_param}
    E_2(\eta) = i\varepsilon(a_k(\eta) - 1)
    \int_0^1 \frac{d\theta}{X_{\theta,1}(\eta)^2},
\end{equation}
and hence,
\begin{equation*}
    \int_{\mathbb{R}^n} E_2(\eta) \phi_x^{\mathrm{in}}(\eta) d\eta
    =
    \int_0^1 \left( \int_{\mathbb{R}^n} \frac{i\varepsilon(a_k(\eta)-1) \phi_x^{\mathrm{in}}(\eta)}
     {X_{\theta,1}(\eta)^2} d\eta
    \right) d\theta.
\end{equation*}
Since $\phi_x^{\mathrm{in}}$ is supported in $\{|\eta| \leq k/2\}$, Lemma \ref{lem:aux_sing} implies that 
\begin{equation*}
    \left|\int_{\mathbb{R}^n} \frac{i\varepsilon(a_k(\eta)-1) \phi_x^{\mathrm{in}}(\eta)}{X_{\theta,1}(\eta)^2} d\eta\right|
    \leq
    C \|i\varepsilon(a_k-1) \phi_x^{\mathrm{in}}\|_{\mathscr{S}_M}
    \leq
    C \|(a_k-1) \phi_x^{\mathrm{in}}\|_{\mathscr{S}_M}.
\end{equation*}
uniformly in $\varepsilon\in (0,1]$, $\theta\in [0,1]$, and $k\geq 1$. Integrating over $\theta \in [0,1]$ gives
\begin{equation*}
    \left| \int_{\mathbb{R}^n} E_2(\eta) \phi_x^{\mathrm{in}}(\eta) d\eta \right| \leq C \|(a_k-1) \phi_x^{\mathrm{in}}\|_{\mathscr{S}_M}
\end{equation*}
Then, \eqref{eq:a_k_phi_leq_phi} and \eqref{eq:phi_in_vs_hat_h} imply that
\begin{equation}\label{eq:E2}
     \sup_{x \in K}\left|\int_{\mathbb{R}^n} E_2(\eta) \phi_x^{\mathrm{in}}(\eta) d\eta \right| \leq \frac{C}{k} \|\widehat h\|_{\mathscr{S}_{M+1}} 
\end{equation}
uniformly in $\varepsilon \in (0,1]$ and $k\geq 1$.

Next, we estimate 
\begin{equation*}
    \langle E_3, \phi_x^{\mathrm{in}} \rangle = \int_{\mathbb{R}^n}
    E_3(\eta) \phi_x^{\mathrm{in}}(\eta) d\eta 
    = 
    \int_{\mathbb{R}^n}
    \left[\frac{1}{X_{0,1}(\eta)} - \frac{1}{X_{0,0}(\eta)}\right]\phi_x^{\mathrm{in}}(\eta) d\eta.
\end{equation*}
Take $A = X_{0,1}$, $B = X_{0,0}$ and note that
\begin{equation*}
    B - A = -\frac{|\eta|^2}{2k},
    \qquad
    \vartheta X_{0,1} + (1-\vartheta) X_{0,0} = X_{0,\vartheta}
\end{equation*}
for $\vartheta\in[0,1]$. Therefore, by \eqref{eq:resolvent_identity_param},
\begin{equation*}
    E_3(\eta) = -\frac{|\eta|^2}{2k}
    \int_0^1 \frac{d\vartheta}{X_{0,\vartheta}(\eta)^2}.
\end{equation*}
and hence,
\begin{equation*}
    \int_{\mathbb{R}^n} E_3(\eta) \phi_x^{\mathrm{in}}(\eta) d\eta
    =
    -\int_0^1 \left( \int_{\mathbb{R}^n} \frac{\frac{|\eta|^2}{2k} \phi_x^{\mathrm{in}}(\eta)} {X_{0,\vartheta}(\eta)^2} d\eta
\right) d\vartheta.
\end{equation*}
Since $\phi_x^{\mathrm{in}}$ is supported in $\{|\eta| \leq k/2\}$, Lemma \ref{lem:aux_sing} implies that 
\begin{equation*}
    \left|\int_{\mathbb{R}^n} \frac{\frac{|\eta|^2}{2k} \phi_x^{\mathrm{in}}(\eta)}{X_{0,\vartheta}(\eta)^2} d\eta\right|
    \leq
    \frac{C}{k} \||\cdot|^2\phi_x^{\mathrm{in}}(\cdot)\|_{\mathscr{S}_M} 
    \leq
    \frac{C}{k} \|\phi_x^{\mathrm{in}}(\cdot)\|_{\mathscr{S}_{M+2}}
\end{equation*}
uniformly in $\varepsilon\in (0,1]$, $\vartheta\in [0,1]$, and $k\geq 1$. Therefore,
\begin{equation*}
    \left|\int_{\mathbb{R}^n} E_3(\eta) \phi_x^{\mathrm{in}}(\eta) d\eta\right| \leq \frac{C}{k} \|\phi_x^{\mathrm{in}}(\cdot)\|_{\mathscr{S}_{M+2}}.
\end{equation*}
Then, \eqref{eq:phi_in_vs_hat_h} implies that
\begin{equation}\label{eq:E3}
     \sup_{x \in K}\left|\int_{\mathbb{R}^n} E_3(\eta) \phi_x^{\mathrm{in}}(\eta) d\eta \right| \leq \frac{C}{k} \|\widehat h\|_{\mathscr{S}_{M+3}}
\end{equation}
uniformly in $\varepsilon\in (0,1]$, $\vartheta\in [0,1]$, and $k\geq 1$.
Combining \eqref{eq:E1}, \eqref{eq:E2}, and \eqref{eq:E3} gives
\begin{equation}\label{eq:final_eps}
    \sup_{\varepsilon\in(0,1]} \sup_{x\in K}
  \left|
    \widetilde{R}_{k,\varepsilon} h^{\mathrm{in}}(x)
    - T_{e_1}^\varepsilon h^{\mathrm{in}}(x)
  \right|
  \leq \frac{C}{k}\|\widehat h\|_{\mathscr{S}_{M}}
\end{equation}
for some $M\in \mathbb{N}$, uniformly on $k> 1$.

It remains to pass from the $\varepsilon$-level estimate \eqref{eq:final_eps} to the estimate at $\varepsilon = 0$.
We write
\begin{multline*}
    \sup_{x \in K}\left|\widetilde{R}_k h^{\mathrm{in}}(x)
    -
    T_{e_1} h^{\mathrm{in}}(x) \right|
    \leq
    \left\|\widetilde{R}_k h^{\mathrm{in}}
    -
    \widetilde{R}_{k,\varepsilon} h^{\mathrm{in}} 
    \right\|_{C^0(K)} 
    + 
    \left\|T_{e_1} h^{\mathrm{in}}
    -
    T_{e_1}^\varepsilon h^{\mathrm{in}} 
    \right\|_{C^0(K)}\\
    + 
    \left\|\widetilde{R}_{k,\varepsilon} h^{\mathrm{in}}
    - T_{e_1}^\varepsilon h^{\mathrm{in}} \right\|_{C^0(K)}.
\end{multline*}
Then, by \eqref{eq:final_eps}, for $\varepsilon\in (0,1]$, we have
\begin{multline*}
    \sup_{x \in K}\left|\widetilde{R}_k h^{\mathrm{in}}(x)
    -
    T_{e_1} h^{\mathrm{in}}(x) \right|
    \leq
    \left\|\widetilde{R}_k h^{\mathrm{in}}
    -
    \widetilde{R}_{k,\varepsilon} h^{\mathrm{in}} 
    \right\|_{C^0(K)} 
    + 
    \left\|T_{e_1} h^{\mathrm{in}}
    -
    T_{e_1}^\varepsilon h^{\mathrm{in}} 
    \right\|_{C^0(K)}\\
    + 
    \frac{C}{k}\|\widehat h\|_{\mathscr{S}_{M}}.
\end{multline*}
For the first term, using \eqref{eq:amp_res_convergence} and Sobolev embedding, we obtain
\begin{equation*}
    \lim_{\varepsilon\to 0} \left\|\widetilde{R}_k h^{\mathrm{in}}
    -
    \widetilde{R}_{k,\varepsilon} h^{\mathrm{in}} 
    \right\|_{C^0(K)} =0.
\end{equation*}
By Lemma \ref{lem:transport}, the second term tends to zero as $\varepsilon\to 0^+$. Hence, taking $\varepsilon\to 0^+$ gives the following. 
\begin{equation}\label{eq:interior_final}
\sup_{x \in K}
\left|
\widetilde{R}_k h^{\mathrm{in}}(x)
-
T_{e_1} h^{\mathrm{in}}(x)
\right|
\le
\frac{C}{k}\,
\|\widehat{h}\|_{\mathscr{S}_{M}}.
\end{equation}
Combining this with \eqref{eq:exterior_final} we complete the proof.
\end{proof}

\section*{Data availability}
No data were used or generated in the research described in this article.

\section*{Declaration of competing interests}
The authors declare that they have no competing interests.

\section*{Declaration of generative AI use}
During the preparation of this work, the authors used Claude and ChatGPT to help simplify and verify certain minor calculations and to improve the clarity and readability of the writing. The authors reviewed and verified the AI-assisted content and take full responsibility for the manuscript.

\bibliographystyle{plain}
\bibliography{references_one_photon}

\setlength{\parskip}{0pt}

\end{document}